\documentclass{amsart}
\usepackage{marktext}
\usepackage{gimac}

\def\b{\mathcal{B}}

\def\p{\partial}

\newcommand{\TV}{\textit{TV}}
\newcommand{\mix}{\textit{mix}}
\newcommand{\diss}{\textit{dis}}
\newcommand{\ka}{\kappa}
\newcommand{\tdis}{t_{\!\diss}}
\newcommand{\tmix}{t_\mix}
\newcommand{\pe}{\mathrm{Pe}}
\newcommand{\ep}{\varepsilon}
\newcommand{\cpl}{\textit{cpl}}
\newcommand{\eff}{\textit{eff}}
\renewcommand{\le}{\leqslant}
\renewcommand{\ge}{\geqslant}

\newcommand{\tcpl}{\tau_\cpl}
\providecommand{\one}[1]{\mathbbm{1}_{#1}}
\providecommand{\eqref}[1]{(\ref{#1})}

\DeclareMathOperator{\leb}{Leb}

\begin{document}
\title[Dissipation Enhancement of Cellular Flows]{Quantifying the dissipation enhancement of cellular flows.}

\author[Iyer]{Gautam Iyer}
\address{%
  Department of Mathematical Sciences,
  Carnegie Mellon University,
  Pittsburgh, PA 15213.}
\email{gautam@math.cmu.edu}
\author[Zhou]{Hongyi Zhou}
\address{%
  Department of Mathematics,
  University of Michigan,
  Ann Arbor, MI 48109.}
\email{hongyizh@umich.edu	}
\thanks{%
  This work has been partially supported by
  the National Science Foundation under
  grants DMS-2108080 to GI, 
  and the Center for Nonlinear Analysis.
}
\subjclass[2020]{Primary 
	35B40; 
	Secondary
	76M45, 
	76R05, 
	37A25. 
}
\keywords{Enhanced dissipation, mixing time, cellular flow}

\begin{abstract}
	We study the dissipation enhancement by cellular flows.
	Previous work by Iyer, Xu, and Zlato\v{s} produces a family of cellular flows that can enhance dissipation by an arbitrarily large amount.
	We improve this result by providing quantitative bounds on the dissipation enhancement in terms of the flow amplitude, cell size and diffusivity.
	Explicitly we show that the \emph{mixing time} is bounded by the exit time from one cell when the flow amplitude is large enough, and by the reciprocal of the effective diffusivity when the flow amplitude is small.
	This agrees with the optimal heuristics.
	We also prove a general result relating the \emph{dissipation time} of incompressible flows to the \emph{mixing time}.
	The main idea behind the proof is to study the dynamics probabilistically and construct a successful coupling. 
\end{abstract}

\maketitle

\section{Introduction}\label{s:intro}

Consider an insoluble dye in an incompressible fluid.
Stirring the fluid typically causes filamentation, stretching blobs of die into fine tendrils.
Diffusion, on the other hand, efficiently damps these small scales, and the combination of these two effects results in \emph{enhanced dissipation} -- the tendency of passive scalars to diffuse faster than in the absence of stirring.
This phenomenon has been extensively studied in many contexts, and various authors have established a link between mixing and dissipation enhancement~\cite{ConstantinKiselevEA08,Zlatos10,FengIyer19,CotiZelatiDelgadinoEA20}, studied dissipation enhancement in more general situations~\cite{Seis20,AlbrittonBeekieEA21,NobiliPottel22} and studied it extensively for shear flows~\cite{Taylor53,BedrossianCotiZelati17,Wei19,GallayZelati21,ColomboCotiZelatiEA21}.
Enhanced dissipation has also been used to suppress non-linear effects arising in certain situations~\cite{FannjiangKiselevEA06,KiselevXu16,FengFengEA20,IyerXuEA21}, and is a subject of active study.

The purpose of this work is to quantify dissipation enhancement for cellular flows, thus providing simple and explicit examples of flows with arbitrarily large dissipation enhancement.
Cellular flows arise as a model problem where ambient fluid velocity is a periodic array of opposing vortices.
They have been extensively studied in the context of fluid dynamics, homogenization and as random perturbations of dynamical systems~\cite{Childress79,ChildressSoward89,FannjiangPapanicolaou94,Koralov04,NovikovPapanicolaouEA05,DolgopyatKoralov08,Bakhtin11,HairerIyerEA18}.

We will use probabilistic techniques to estimate the \emph{mixing time} of a diffusion whose drift is a cellular flow.
We then estimate the dissipation enhancement in terms of the mixing time.
The bounds we obtain are significantly better than the bounds previously obtained in~\cite{IyerXuEA21}, and (up to a logarithmic factor) they agree with the optimal heuristic bounds.

\subsection*{Acknowledgements}

We thank Andrej Zlato\v s and the anonymous referee for helpful comments that led to an improvement of the main result when $A \leq \kappa/\epsilon^4$.

\section{Main Result}

We will study the concentration of a dye, denoted by $\phi$, as a passive scalar, evolving according to the advection diffusion equation
\begin{equation}\label{e:phiEq}
	\p_t \phi - (u \cdot \grad) \phi - \frac{\ka}{2} \lap \phi = 0 \qquad \text{in } (0, \infty) \times \T^d\,.
\end{equation}
Here $-u$ represents the velocity field of the ambient fluid, and $\kappa/2 > 0$ is the molecular diffusivity.
We restrict our attention to the periodic $d$-dimensional torus $\T^d$ with side length~$1$, and we will normalize the initial concentration, $\phi_0$, so that
\begin{equation*}
\int_{\T^d} \phi_0(x) \, dx = 0\,.
\end{equation*}

As time evolves, the dye spreads uniformly across the torus and $\phi(\cdot, t) \to 0$ as~$t \to \infty$.
One measure of convergence rate that will interest us is the~\emph{dissipation time}: the time required for solutions to~\eqref{e:phiEq} to lose a constant fraction of their initial energy (see for instance~\cite{FannjiangNonnenmacherEA04,ConstantinKiselevEA08,FengIyer19}).
Explicitly, \emph{dissipation time}, denoted by $\tdis = \tdis(\kappa, u)$ is defined by
\begin{equation}\label{e:tdisDefPDE}
	\tdis \defeq \inf\set[\Big]{t \geq 0 \st \norm{\phi(s + t)}_{L^2} \leq \frac{1}{2} \norm{\phi(s)}_{L^2}\,~ \text{for all } s \geq 0,~ \phi(s) \in \dot L^2}\,.
\end{equation}
Here $\dot L^2$ denotes the space of all mean-zero, square integrable functions on the torus~$\T^2$.

The Poincar\'e inequality and the fact that $u$ is divergence free immediately imply
\begin{equation}\label{e:poincare}
\tdis(\kappa, u) \leq \frac{1}{4 \pi^2 \kappa}\,.
\end{equation}
However, this is only an upper bound, and the dissipation time may in fact be much smaller than $O(1/\kappa)$.
When this occurs (i.e.\ when $\tdis(u, \kappa) \leq o(1/\kappa)$) it is known as \emph{enhanced dissipation}.
Intuitively, enhanced dissipation when the stirring velocity field generates small scales (e.g. through filamentation), which are then damped much faster by the diffusion.

Seminal work of Constantin et\ al.~\cite{ConstantinKiselevEA08} provides a spectral characterization of (time independent) velocity fields for which $\tdis = o(1/\kappa)$.
More explicit, improved bounds were recently obtained in terms of the mixing rate of~$u$.
For instance, if~$u$ is exponentially mixing then one can show $\tdis \leq O(\abs{\ln \kappa}^2)$ (see for instance~\cite{FengIyer19,Feng19,CotiZelatiDelgadinoEA20}).

In the context of applications, various authors have shown that sufficiently enhanced dissipation can be used to quench reactions, stop phase separation and prevent singularity formation (see for instance~\cite{FannjiangKiselevEA06,KiselevXu16,FengFengEA20,IyerXuEA21,FengMazzucato22,FengShiEA22}).
Thus finding simple and explicit examples of flows which sufficiently enhance dissipation (i.e.\ make~$\tdis$ arbitrarily small) are useful for many applications.
While such flows can be found by rescaling velocity fields with strong enough mixing properties (see for instance~\cites{FengFengEA20,IyerXuEA21}), examples of mixing velocity fields on the torus are notoriously hard to construct.
The main goal of this work is to provide a simple and explicit family of velocity fields for which~$\tdis$ can be made arbitrarily small. 
The family of flows we construct are two dimensional cellular flows.
These arise frequently in fluid dynamics as flows around strong arrays of opposing vortices and have been extensively studied~\cite{Childress79,RhinesYoung83,ChildressSoward89,FannjiangPapanicolaou94,Heinze03,NovikovPapanicolaouEA05,Koralov04}.

Given $\epsilon > 0$, consider the cellular flow~$v$ defined by
\begin{equation}\label{e:v}
v \defeq \grad^\perp (\xi H)
= \begin{pmatrix} - \partial_2 (\xi H) \\ \phantom-\partial_1 (\xi H) \end{pmatrix}
\,,
\qquad\text{where }
H(x) \defeq
\sin\paren[\Big]{ \frac{2 \pi x_1}{\epsilon} }
\sin\paren[\Big]{ \frac{2 \pi x_2}{\epsilon} }\, ,
\end{equation}
and $\xi$ is a smooth periodic cutoff function function such that
\begin{equation*}
	\xi(x) = \begin{cases}
		1 & \abs{H(x)} \leq \frac{1}{4}\,,\\
		0 & \abs{H(x)} \geq \frac{1}{2}\,.
	\end{cases}
\end{equation*}

\begin{figure}[htb]
\centering
\includegraphics[width=.5\textwidth]{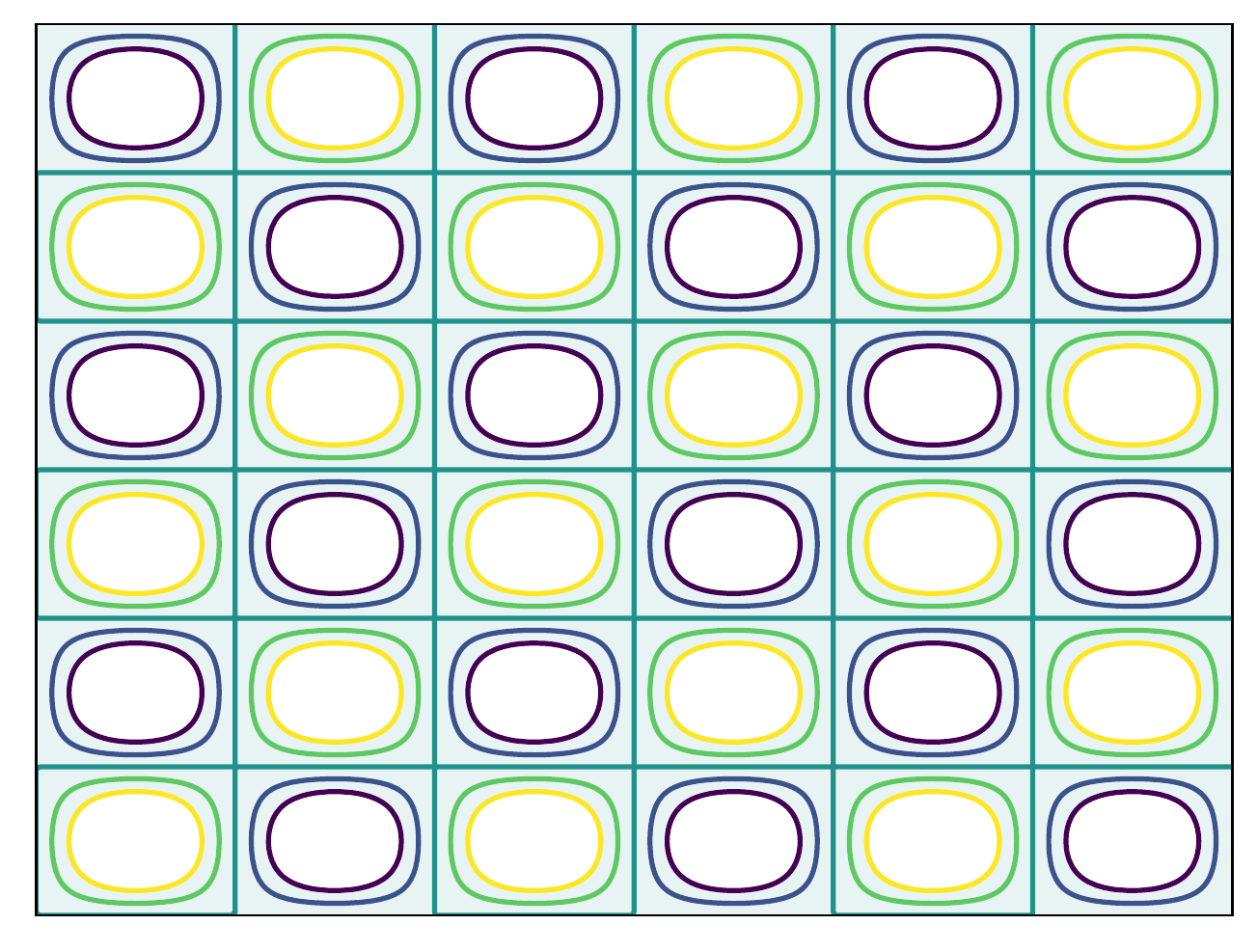}
\caption{%
	Stream lines of the cellular flow defined in equation~\eqref{e:v}.
	The flow is only non-zero in the shaded region.
}
\label{f:cflow}
\end{figure}
This flow has cell size~$O(\epsilon)$, and its stream lines are shown in Figure~\ref{f:cflow}.
Our main result chooses $u = A v$ for $A$ large, and estimates the \emph{mixing time} explicitly in terms of the flow amplitude~$A/\epsilon$, cell size~$\epsilon$ and diffusivity~$\kappa$ as follows.

\begin{theorem}\label{t:main}
	Suppose
	\begin{equation}
		\frac{\epsilon^2}{\kappa} \ll 1\,,\quad
		A \gg \frac{\kappa}{\epsilon^2} \,,
		\quad
		\delta = \sqrt{\frac{\kappa}{A}}\,,
		\quad
		u = A v\,,
	\end{equation}
	where $v$ is defined in~\eqref{e:v}.
	Then there exists a finite constant $C$, independent of $\epsilon$, $A$, and $\kappa$, such that 
	\begin{equation}\label{e:tdisBound}
		\tdis \leq 3 \tmix
		\leq \begin{dcases}
			\frac{C \epsilon^2}{\kappa}
				& A \geq \frac{\kappa \abs{\ln \delta}^2}{\epsilon^4}\,,
			\\
			C\paren[\Big]{ \frac{\epsilon^2}{\kappa} + \frac{\abs{\ln \delta}^2}{\epsilon^2 A} }
				& \frac{\kappa}{\epsilon^4} \leq A \leq \frac{\kappa \abs{\ln \delta}^2}{\epsilon^4}\,,
			\\
			\frac{C}{\sqrt{\kappa A}}
				& \frac{\kappa}{\epsilon^2} \ll A \leq \frac{\kappa}{\epsilon^4}\,.
		\end{dcases}
	\end{equation}
\end{theorem}

\begin{@empty}
Here $\tmix = \tmix(u, \kappa)$ is the \emph{mixing time}, a notion that we describe in Section~\ref{s:mixingtime}, below.
We first compare Theorem~\ref{t:main} to the well known homogenization results that estimate the \emph{effective diffusivity}.
Recall standard results (see for instance~\cites{BensoussanLionsEA78,PavliotisStuart08}) show that the long time behavior of~\eqref{e:phiEq} is effectively that of the purely diffusive equation
\begin{equation}\label{e:DtBarPhi}
	\partial_t \bar \phi - \frac{1}{2} D_\eff \lap \bar \phi = 0\,,
\end{equation}
with an enhanced diffusion coefficient $D_\eff$, known as the \emph{effective diffusivity}.
The effective diffusivity of cellular flows has been extensively studied~\cite{Childress79,ChildressSoward89,FannjiangPapanicolaou94,Koralov04} and is known to asymptotically be
\begin{equation}
	D_\eff \approx O\paren{ \sqrt{\kappa A} }\,,
\end{equation}
as $\kappa \to 0$ (with $A$ fixed), or $A \to \infty$ (with $\kappa$ fixed).
Given this one would expect from~\eqref{e:poincare} that
\begin{equation}\label{e:tdisDeff}
	\tdis \approx \frac{1}{4 \pi^2 D_\eff} = O\paren[\Big]{\frac{1}{\sqrt{\kappa A}} }\,,
\end{equation}
and this is exactly~\eqref{e:tdisBound} when $A \leq \kappa / \epsilon^4$.

The reason one has a different bounds depending on the relative size of $A$ and  $\kappa / \epsilon^4$ is as follows.
One can consider the simultaneous limit of~\eqref{e:phiEq} as $\epsilon, \kappa \to 0$, $A \to \infty$.
In this case one can show that~$\phi$ either homogenizes, and behaves like the solution to the effective equation~\eqref{e:DtBarPhi}, or averages along stream lines and can be described by a diffusion on a Reeb graph~\cite{FreidlinWentzell12,PavliotisStuart08}.
This transition occurs precisely at $A \approx \kappa / \epsilon^4$, and was studied previously in~\cite{IyerKomorowskiEA14,HairerKoralovEA16,HairerIyerEA18}), and this explains the condition $A \leq \kappa / \epsilon^4$ in~\eqref{e:tdisBound}.

As explained earlier, when $A \leq \kappa / \epsilon^4$ the problem homogenizes and the upper bound in~\eqref{e:tdisBound} is consistent with the bound~\eqref{e:tdisDeff} obtained from homogenization.
When $A \geq \kappa / \epsilon^4$, the upper bound~\eqref{e:tdisDeff} can not hold.
Indeed, in one cell, movement in the direction transverse to stream lines of~$u$ happens through diffusion alone.
Thus the time taken for a dye to diffuse across one cell is at least $\epsilon^2 / \kappa$, and so we must have $\tdis \geq C \epsilon^2 / \kappa$.
Of course $A \geq \kappa / \epsilon^4$ is equivalent to $\epsilon^2 / \kappa \geq 1/\sqrt{\kappa A}$, and so~\eqref{e:tdisDeff} can not hold.

In the proof of Theorem~\ref{t:main} we will in fact show
\begin{equation}\label{e:tmixAveraging}
	\tmix \leq C\paren[\Big]{ \frac{\epsilon^2}{A} + \frac{ \abs{\ln \delta}^2 }{\epsilon^2 A} }\,,
\end{equation}
for all~$A \gg \kappa / \epsilon^2$.
This is of course weaker than~\eqref{e:tdisDeff} when $A \leq \kappa / \epsilon^4$, but better when $A \geq \kappa / \epsilon^4$.
Moreover, when $A$ is sufficiently large the second term is dominated by the first one, which is the bound stated in~\eqref{e:tdisBound}.
We will provide a quick heuristic explanation for~\eqref{e:tmixAveraging} later in this section.
\medskip

One application for Theorem~\ref{t:main} is to produce flows with a small dissipation time.
From~\eqref{e:tdisBound} we see that for fixed~$A, \kappa$, the choice of~$\epsilon$ that minimizes $\tmix$ is
\begin{equation*}
	\epsilon = \paren[\Big]{\frac{\kappa}{A}}^{1/4}\,.
\end{equation*}
This choice of~$\epsilon$ leads to
\begin{equation}\label{e:tMixE2Ka}
	\tmix
		\leq \frac{C\epsilon^2}{\kappa}
		\,,
\end{equation}
which is time taken to diffuse through one cell.
By choice of $\epsilon$, we have $\epsilon^2 / \kappa = 1 / \sqrt{\kappa A}$, which vanishes as $A \to \infty$.
This provides a simple family of explicit flows with arbitrarily small (and explicit) dissipation time.
\end{@empty}

We 
note that the first author, Xu and Zlato\v s~\cite{IyerXuEA21} have already shown that that the dissipation time of a sufficiently strong and fine cellular flow can be made arbitrarily small.
The estimates in~\cite{IyerXuEA21}, however, are neither explicit nor optimal.
In particular, Theorem 1.3 in~\cite{IyerXuEA21} only asserts the existence of sufficiently strong and fine cellular flows with arbitrarily small~$\tdis$, without providing a quantitative bound.
A more explicit bound is provided in~\cite[Remark 6.6]{IyerXuEA21} which  yields a sub-optimal bound of the form $\tdis \leq C \log(A/\kappa) A^{-1/64} \kappa^{-1}$ after rescaling.
This is much weaker than~\eqref{e:tdisBound}, or the explicit $\epsilon^2 / \kappa$ bound described above.

\subsection{The mixing time}\label{s:mixingtime}

We now define the mixing time $\tmix$ appearing in~\eqref{e:tdisBound}.
This is typically used in probability to measure the rate convergence of Markov processes~\cite{LevinPeresEA09,MontenegroTetali06} to their stationary distribution.
In our case, the \emph{mixing time} is the minimum amount of time required for the fundamental solution of~\eqref{e:phiEq} to be $L^1$-close to the constant function~$1$.
That is, if $\rho(x, s; y, t)$ is the fundamental solution of~\eqref{e:phiEq}, the mixing time is defined by
\begin{equation}\label{e:tmixDef}
  \tmix \defeq \inf\set[\Big]{ t \geq 0 \st \sup_{x \in \T^d,\; s \geq 0} \int_{\T^d} \abs{ \rho(x, s; y, s+t) - 1} \, dy < \frac{1}{2}\,,~\forall s\geq 0 }\,.
\end{equation}

The mixing time and dissipation time are related to each other: the dissipation time is bounded by three times the mixing time.
The mixing time can also be bounded by the dissipation time, up to a logarithmic factor.
This is a general result and is not specific to cellular flows.

\begin{proposition}\label{p:tMixTDis}
	Let $u \in L^\infty( [0, \infty); W^{1, \infty}(\T^d) )$ be a divergence free vector field, and let $\tmix = \tmix(u, \kappa)$, $\tdis = \tdis(u, \kappa)$ denote the mixing time and dissipation time respectively.
  There exists a dimensional constant $C = C(d) < \infty$, independent of $u$ and $\kappa$ such that for all sufficiently small~$\kappa > 0$ we have
	\begin{equation}\label{e:disMix}
		\tdis
			\leq 3 \tmix
			\leq C \tdis \ln \paren[\Big]{ 1 + \frac{1}{\ka \tdis} }\,.
	\end{equation}
\end{proposition}
\begin{remark}
	By rescaling we see that on a torus with side length~$\ell$, the above becomes
	\begin{gather}\label{e:disMixEll}\tag{\ref*{e:disMix}$'$}
		\tdis
			\leq 3 \tmix
			\leq C \tdis \ln \paren[\Big]{ 1 + \frac{\ell^2}{\ka \tdis} }\,.
	\end{gather}
	for some dimensional constant $C=C(d)$ that is independent of~$\ell$.
\end{remark}

We are presently unaware whether or not the logarithmic factor is necessary.
We prove Proposition~\ref{p:tMixTDis} in Section~\ref{s:tdismix}, below.

\subsection{The main idea behind the proof.}

We now provide a non-technical description of the main idea behind the proof of Theorem~\ref{t:main}.
The Ito diffusion associated to~\eqref{e:phiEq} is defined by the SDE
\begin{equation}\label{eq:itodef}
	d X_t = A v(X_t) \, dt + \sqrt{\kappa} \, d B_t\,,
\end{equation}
on the $2$-dimensional torus $\T^2$.
Here $B$ is a standard $2$-dimensional Brownian motion.
Since $\dv v = 0$, the invariant measure of~$X$ is the Lebesgue measure on the torus.

\begin{@empty}
	To estimate the mixing time, let us first heuristically study the time taken for $X$ to start from a point~$x$ and reach a given point~$y$.
	To do this,~$X$ has to first exit the cell containing~$x$.
	Since movement transverse to stream lines occurs through diffusion alone, the process~$X$ will take $O(\epsilon^2 / \kappa)$ to exit a cell.
	After exiting this cell, the process~$X$ needs to explore the torus until it reaches the boundary of the cell containing~$y$.
	During this phase, the process~$X$ essentially performs one step of a random walk on the lattice of $1/\epsilon^2$ cells every time it crosses a boundary layer of thickness $\epsilon \delta$.
	(The thickness~$\epsilon \delta$ is chosen so that the time taken for~$X$ to cross the boundary layer through diffusion is comparable to the time taken for it travel around the boundary layer through convection.)
	Since the mixing time of a random walk on a 2D lattice of $1/\epsilon^2$ points is $O(1/\epsilon^2)$, the mixing time of~$X$ should be $O(\epsilon^2 / \kappa + T_\mathrm{cross})$, where $T_\mathrm{cross}$ is the expected time required to make $O(1/\epsilon^2)$ boundary layer crossings.

	We can (heuristically) estimate $T_\mathrm{cross}$ as follows.
	Since each boundary layer crossing happens through diffusion alone, $T_\mathrm{cross}$ should be comparable to $\tilde T_\mathrm{cross}$, where $\tilde T_\mathrm{cross}$ is the expected time taken for~$\tilde B^\kappa$ to make $1/\epsilon^2$ crossings over the interval $[0, \epsilon \delta]$.
	Here $\tilde B^\kappa$ is a doubly reflected Brownian motion on the interval $[0, \epsilon]$ with diffusivity~$\kappa$.

	On time scales smaller than $\epsilon^2/\kappa$, the process $\tilde B^\kappa$ won't feel the reflection at the right boundary $\epsilon$.
	Thus if $\tilde T_\mathrm{cross} \ll \epsilon^2 / \kappa$, then $\tilde T_\mathrm{cross}$ should be comparable to the time taken for a standard Brownian motion to make $O(1/\epsilon^2)$ crossings of the interval $[0, \epsilon \delta / \sqrt{\kappa}]$.
	This quickly shows~$\tilde T_\mathrm{cross} = O(1/(\epsilon^2 A))$.
	Of course, when $A \geq \kappa / \epsilon^4$,  $1/(\epsilon^2 A) \leq \epsilon^2 / \kappa$, and so $\tilde T_\mathrm{cross} \leq \epsilon^2 / \kappa$.

	On time scales larger than $\epsilon^2 / \kappa$, the process $\tilde B^k$ mixes on the interval $[0, \epsilon \delta]$.
	The number of boundary layer crossings in time $T$ will become proportional to the ratio of the time $\tilde B^\kappa$ spends in the boundary layer to the expected exit time from the boundary layer.
	Using this we can check
	$\tilde T_\mathrm{cross} = O(1 / \sqrt{\kappa A})$.

	This heuristic is what leads to Theorem~\ref{t:main}.
	Moreover, the above heuristic suggests a lower bound of the form  $O(\epsilon^2 / \kappa + T_\mathrm{cross})$, and hence the bounds in Theorem~\ref{t:main} should be optimal.
	We will make the above heuristic rigorous by constructing a \emph{successful coupling} of the process~$X$ (described in Section~\ref{sec:main}, below).
\end{@empty}

Before delving into the details we make three remarks:
First, the extra logarithmic factor $\abs{\ln \delta}$ in~\eqref{e:tdisBound} arises due to the logarithmic slow down of Hamiltonian systems as they approach hyperbolic critical points (all cell corners, in our case).
Second, the smooth cutoff~$\xi$ in~\eqref{e:v} is used to initiate the coupling of the projected processes in a time that is independent of~$A$.
Third, the explicit formula for~$H$ in~\eqref{e:v} is used to construct a simple coupling in subsequent steps using symmetry.
While the logarithmic factor~$\abs{\ln \delta}$ is unavoidable, both the smooth cutoff~$\xi$ and the explicit formula for~$H$ are mainly used to simplify technicalities in the proof.

\subsection*{Plan of this paper}
In the next section (Section~\ref{sec:main}) we prove Theorem~\ref{t:main}, modulo several technical lemmas bounding certain hitting times.
In Section~\ref{s:tdismix} we prove Proposition~\ref{p:tMixTDis}, relating the dissipation time and mixing time for general incompressible flows.
In Section~\ref{s:cellCoupling} we prove an $O(\epsilon^2 / \kappa)$ bound on the coupling time when the process is projected to a torus of side length~$\epsilon$.
Finally in Section~\ref{s:phase2} we prove the remaining lemmas stated in Section~\ref{sec:main} by counting boundary layer crossings.

\section{Proof of the Mixing Time Bound (Theorem~\ref{t:main})}\label{sec:main}

The goal of this section is to prove Theorem~\ref{t:main}.
In light of Proposition~\ref{p:tMixTDis}, we only need to bound the mixing time.
We will do this by a \emph{coupling construction}.
To fix notation, we will subsequently assume~$X$ and~$\tilde X$ are solutions of the SDEs
\begin{align}
  \label{e:sdeX}
  d X_t &= A v(X_t) dt + \sqrt{\ka} \, d B_t \,,
  \\
  \label{e:sdeXtilde}
  d \tilde{X}_t &= A v(\tilde{X}_t) dt + \sqrt{\ka} \, d \tilde{B}_t \,.
\end{align}
with initial data
\begin{equation*}
  X_0 = x\,,
  \quad
  \tilde X_0 = \tilde x
  \qquad
  \P^{(x, \tilde x)} \text{-almost surely}\,.
\end{equation*}
Here $B$ and $\tilde{B}$ are both 2D Brownian motions.
We will choose~$\tilde B$ in terms of $B$ in a manner that ensures a suitable bound on the \emph{coupling time}.
Recall the coupling time
\begin{equation*}
  \tcpl \defeq \inf \set{ t \geq 0 \st X_t = \tilde X_t }\,,
\end{equation*}
is the first time $X$ and $\tilde X$ meet, and standard results (see for instance~\cite[Ch.~5]{LevinPeresEA09}) guarantee
\begin{equation}\label{e:tMixTCouple}
  \tmix \leq C \sup_{(x, \tilde x) \in \T^2 \times \T^2} \E^{(x, \tilde x)} \tcpl\,.
\end{equation}

Thus our task is now to choose the Brownian motion $\tilde B$ and bound $\E\tcpl$.
The construction of~$\tilde B$ can be described quickly, however, the bound on~$\E\tcpl$ requires several technical lemmas.
Moreover, the proof when $A \geq \kappa / \epsilon^4$ differs from the proof when $A \leq \kappa / \epsilon^4$ differ in only one aspect -- the estimate of the coupling time.
For clarity of presentation we will describe the construction of~$\tilde B$ below assuming $A \geq \kappa / \epsilon^4$, and momentarily postpone the case when $A \leq \kappa / \epsilon^4$ and the lemmas bounding the coupling time.

\begin{proof}[Proof of Theorem~\ref{t:main} when $A \geq \kappa / \epsilon^4$]
  The coupling construction is divided into several stages, which we describe individually.
  \restartsteps
  \step[Coupling projections]
  Observe first that the drift~$v$ is periodic with period~$\epsilon$, and thus both~$X$ and~$\tilde X$ can be viewed as diffusions on a torus with side length~$\epsilon$.
  Let $\T^2_\epsilon = [0, \epsilon)^2$ be the two dimensional torus with side length~$\epsilon$, and let $\Pi_\epsilon \colon \T^2 \to \T^2_\epsilon$ be the projection defined by
  \begin{equation*}
    \Pi_\epsilon(x_1, x_2) = \paren[\big]{ x_1 \pmod \epsilon,~x_2 \pmod \epsilon}\,.
  \end{equation*}
  We will subsequently assume that $1/\epsilon \in \N$, so the above projection is well defined.
  (We also clarify that $\T^2$ above denotes the standard two dimensional torus with side length~$1$.)

  Consider the projected diffusions
  \begin{equation}\label{e:Ydef}
    Y = \Pi_\epsilon X\,, \qquad \tilde Y = \Pi_\epsilon \tilde X\,,
  \end{equation}
  on the torus $\T^2_\epsilon$.
  Since the drift~$v$ is divergence free one can use PDE methods to show that the mixing time of $Y$ is bounded by $O(\epsilon^2/ \kappa)$.
  This, however, is not sufficient for our purposes as we need a coupling between~$Y$ and $\tilde Y$ for subsequent steps, and we need the coupling time to be bounded independent of~$A$.
  We will couple~$Y$ and~$\tilde Y$ by waiting until they enter the central region of cells where $u = 0$.
  In this region~$Y$ and~$\tilde Y$ are simply Brownian motions, and we can couple them by reflection (see for instance~\cite{LindvallRogers86}), in time~$\tcpl^Y$ that is bounded independent of~$A$.
	Explicitly, we will show (Lemma~\ref{l:cellCoupling}, below) that
  \begin{equation}\label{e:ycouple}
    \E^{(x, \tilde x)} \tcpl^Y \leq \frac{C \epsilon^2}{\kappa} \,,
  \end{equation}
  for some finite constant $C$.
  Here, and subsequently, we will assume that the constant $C$ is independent of the parameters $\epsilon$, $A$, $\kappa$, the initial data $x, \tilde x$, and may increase from line to line.

  \step[Moving to vertical cell boundaries]
  By the Markov property, we may now restart time and assume that at time $0$ we have $\Pi_\epsilon X_0 = \Pi_\epsilon \tilde X_0$.
  In this step we will now choose $B = \tilde B$ and wait until $X$ and $\tilde X$ hit a vertical cell boundary.
  That is, we set
  \begin{equation}\label{e:sigmavDef}
    \sigma_v \defeq \inf \set[\Big]{ t \geq 0  \st X^1_t \in \frac{\epsilon}{2} \Z}\,,
    \qquad
    \tilde{\sigma}_v \defeq \inf \set[\Big]{ t \geq 0  \st \tilde X^1_t \in \frac{\epsilon}{2} \Z}\,,
  \end{equation}
  where $X^1, \tilde X^1$ denote the first coordinates process of $X$, and $\tilde X$ respectively.
  Periodicity of~$v$ will ensure~$\sigma_v = \tilde \sigma_v$, and we will show (Lemma~\ref{l:vhit}, below) that
  \begin{equation}\label{e:vhit}
		\E^x \sigma_v \leq \frac{C \epsilon^2 }{\kappa}
  \end{equation}

	\begin{figure}[htb]
		\centering
		\includegraphics[width=.7\textwidth]{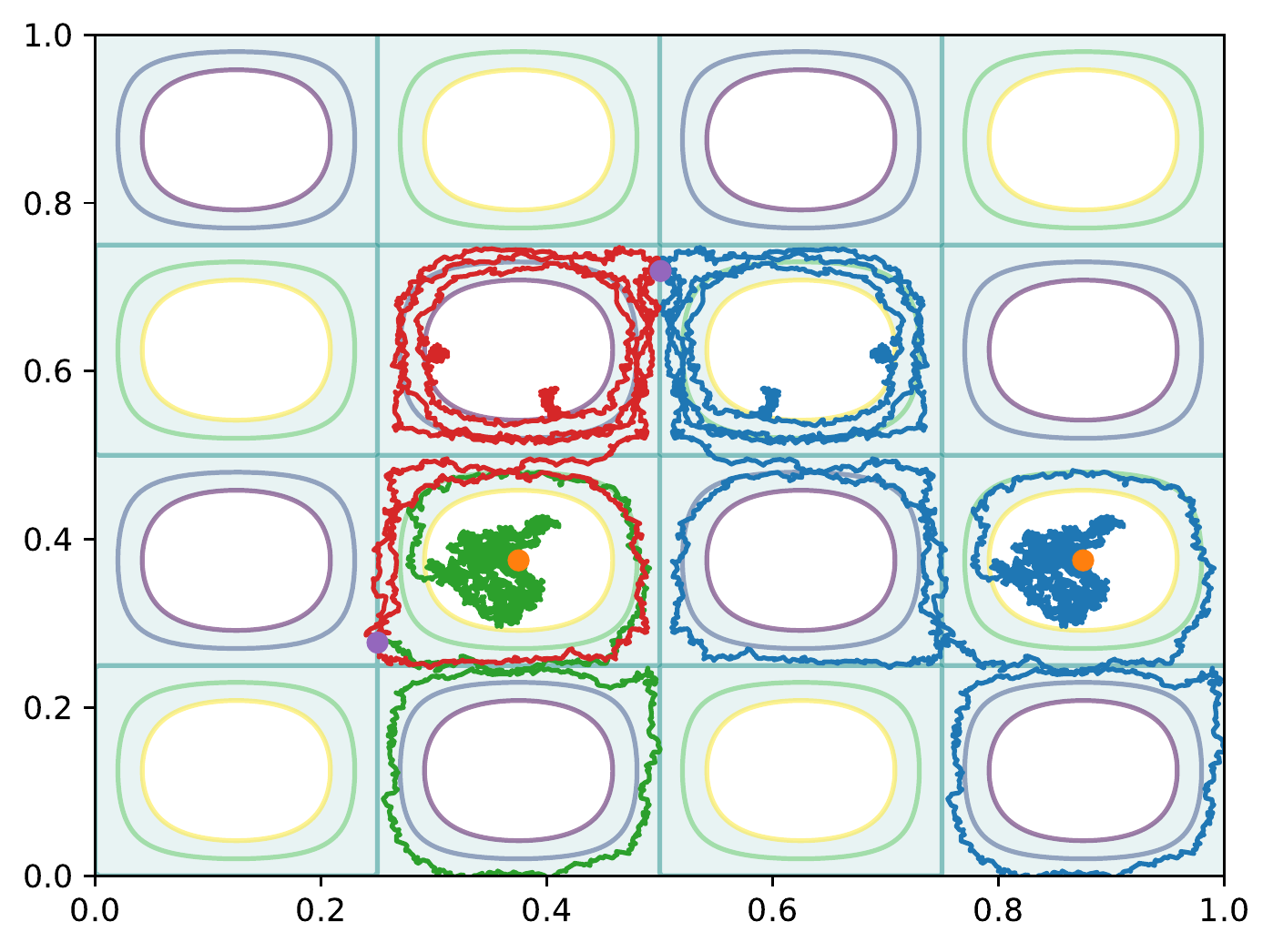}
		\caption{%
			Sample trajectories illustrating the coupling in steps 2 and~3.
			Here $X_0 = (0.75, 0.25)$, $\tilde X_0 = (0.25, 0.25)$, and the trajectory of~$X$ is shown in blue.
			Until $\tilde X$ hits a vertical cell boundary the trajectory of~$\tilde X$ (shown in green) is simply a shift of the trajectory of~$X$.
			After this time, the trajectory of~$\tilde X$ (shown in red) is a mirror image of the trajectory of~$X$ until they hit the same vertical line ($x = 0.5$ in this case).
		}
		\label{f:traject}
	\end{figure}

  \step[Vertical Coupling]
  By the Markov property again, we restart time and assume $\Pi_\epsilon X_0 = \Pi_\epsilon \tilde X_0$, and $X^1_0, \tilde X^1_0 \in \frac{\epsilon}{2} \Z$.
  We will now choose $\tilde B^1 = -B^1$ and $\tilde B^2 = B^2$, and wait until time $\tau_v$ defined by
  \begin{equation}\label{e:tauvDef}
    \tau_v \defeq \inf\set{t \geq 0 \st X^1_t = \tilde X^1_t}\,.
  \end{equation}
  Note that by symmetry of $v$ we will have%
  \footnote{We clarify here that $(\Pi_\epsilon X_t)^2$ refers to the second coordinate of $\Pi_\epsilon X_t$.}
  $(\Pi_\epsilon X_t)^2 = (\Pi_\epsilon \tilde X_t)^2$ for all $t \leq \tau_v$, and thus at time $\tau_v$ we will have~$\Pi_\epsilon X_t = \Pi_\epsilon \tilde X_t$.
	(See Figure~\ref{f:traject}, below, for an illustration of trajectories of~$X$ and~$\tilde X$ under this choice of noise.)
  We will show (Lemma~\ref{l:vcpl}, below) that
  \begin{equation}\label{e:vcpl}
    \E \tau_v \leq  \dfrac{C \abs{\ln \delta}^2}{A \ep^2}  \,.
  \end{equation}
  The proof of~\eqref{e:vcpl} requires a estimates on the number of times the flow crosses the boundary layer; this is technical, but has been well studied by numerous authors and the proofs can be readily adapted to our situation.

  \step[Horizontal hitting and coupling]
  At this point we have arranged for $\Pi_\epsilon X = \Pi_\epsilon \tilde X$, and $X^1 = \tilde X^1 \in \frac{\epsilon}{2} \Z$.
  As usual, we restart time and assume that the above happens at time $0$.
  We will now repeat steps 2 and 3 in the horizontal direction:
    First choose $\tilde B = B$ until $X^2,  \tilde X^2 \in \frac{\epsilon}{2} \Z$, then choose $\tilde B^1 = B^1$, $\tilde B^2 = - B^2$, and then wait until $X^2 = \tilde X^2$.
  The time taken for each of these steps is bounded in Lemmas~\ref{l:hhit} and~\ref{l:hcpl}, below.
  The symmetry of $v$ will ensure that when $X^2 = \tilde X^2$, we will also have $X^1 = \tilde X^1$, thus giving a successful coupling of $X, \tilde X$.
  \smallskip

  Using Chebychev's inequality, the above guarantees us a coupling of $X, \tilde X$ with probability at least $1/16$ in time at most twice the expected value of the stopping times in each of the above steps.
  Thus using the Markov property and Lemmas~\ref{l:cellCoupling}--\ref{l:hcpl}, below, we obtain a successful coupling with the coupling time bounded by
  \begin{equation}\label{e:couplingBound}
    \E^{(x, \tilde x)} \tau_\cpl \leq C \paren[\Big]{ \dfrac{\ep^2 }{\ka} + \dfrac{\abs{\ln \delta}^2}{A \ep^2}  }\,.
	\end{equation}
  Using~\eqref{e:tMixTCouple}, proves~\eqref{e:tmixAveraging} which proves~\eqref{e:tdisBound} when $A \geq \kappa/\epsilon^4$.
\end{proof}

It remains to bound the stopping times in each of the above steps.
For clarity of presentation we state each bound as a lemma below, and prove the lemmas in subsequent sections.
We assume $A \ge \ka / \ep^4$ through the following five lemmas. 

\begin{lemma}[Coupling of projections]\label{l:cellCoupling}
  There exists a Brownian motion~$\tilde{B}$ such that $(Y, \tilde{Y})$ is a coupling of $Y$ (on the torus $\T^2_\epsilon$), and the coupling time satisfies~\eqref{e:ycouple}.
\end{lemma}

\begin{lemma}[Vertical boundary hitting time]\label{l:vhit}
  Suppose $\Pi_\epsilon X_0 = \Pi_\epsilon \tilde X_0$.
  Choose $\tilde B = B$, and let~$\sigma_v$ and $\tilde \sigma_v$ (equation~\eqref{e:sigmavDef}) be the first hitting times of $X$ and $\tilde X$ to the vertical cell boundaries, respectively.
  Then $\sigma_v = \tilde \sigma_v$ and equation~\eqref{e:vhit} holds.
\end{lemma}

\begin{lemma}[Vertical coupling]\label{l:vcpl}
  Suppose $\Pi_\epsilon X_0 = \Pi_\epsilon \tilde X_0$, and $X_0^1 \in \frac{\epsilon}{2} \Z$.
  Let $\tilde B = (-B^1, B^2)$, and let~$\tau_v$ (equation~\eqref{e:tauvDef}) be the first time first time $X$ and $\tilde X$ are on the same vertical line.
  Then~$(\Pi_\epsilon X_t)^2 = (\Pi_\epsilon \tilde X_t)^2$ for all $t \leq \tau_v$, the expected value of~$\tau_v$ is bounded by~\eqref{e:vcpl}.
\end{lemma}

\begin{lemma}[Horizontal boundary hitting time]\label{l:hhit}
  Suppose that $\Pi_\epsilon X_0 = \Pi_\epsilon \tilde X_0$, and $X_0^1 = \tilde X^1_0 \in \frac{\epsilon}{2} \Z$.
  Choose $\tilde B = B$, and let
  \begin{equation*}
    \sigma_h \defeq \inf \set[\Big]{ t \geq 0  \st X^2_t \in \frac{\epsilon}{2} \Z}\,,
    \qquad
    \tilde \sigma_h \defeq \inf \set[\Big]{ t \geq 0  \st \tilde X^2_t \in \frac{\epsilon}{2} \Z}\,,
  \end{equation*}
  be the first hitting time to the horizontal cell boundaries.
  Then
  \begin{equation*}
    \sigma_h = \tilde \sigma_h\,,
    \qquad
    X^1_{\sigma_h} = \tilde X^1_{\tilde \sigma_h}\,,
    \qquad\text{and}\qquad
    \E \sigma_h
    = \E \tilde \sigma_h
    \leq \frac{C \epsilon^2}{\kappa}\,.
  \end{equation*}
\end{lemma}

\begin{lemma}[Horizontal coupling]\label{l:hcpl}
  Suppose $\Pi_\epsilon X_0 = \Pi_\epsilon \tilde X_0$, $X_0^1 = \tilde X^1_0$, and $X^2_0 \in \frac{\epsilon}{2} \Z$.
  Choose $\tilde B = (B^1, -B^2)$ and let
  \begin{equation*}
    \tau_h = \inf \set[\Big]{t \geq 0 \st X^2_t = \tilde X^2_t} \,,
  \end{equation*}
  be the first time $X$ and $\tilde X$ are on the same horizontal line.
  Then 
  \begin{equation}\label{e:tauHbound}
    \tau_h = \tilde \tau_h\,,
    \qquad
    X_{\tau_h} = \tilde X_{\tau_h}\,,
    \qquad\text{and}\qquad
    \E \tau_h \le \dfrac{C \abs{\ln \delta}^2}{A \ep^2} \,.
  \end{equation}
\end{lemma}

Each of these lemmas will be proved in subsequent sections.
Finally, we conclude this section by stating the modifications necessary to prove Theorem~\ref{t:main} in the case where $A \leq \kappa / \epsilon^4$.
\begin{proof}[Proof of Theorem~\ref{t:main} when $A \leq \frac{\kappa}{\epsilon^4}$]
	The coupling construction is identical to the construction in the case where $A \geq \frac{\kappa}{\epsilon^4}$.
	The only differences are the bounds on the hitting time: 
	\begin{enumerate}\reqnomode
	  \item 
			In Lemma~\ref{l:vcpl}, the vertical hitting time bound~\eqref{e:vcpl} should be replaced by
			\begin{equation}\label{e:vcplPrime}
				\tag{\ref{e:vcpl}$'$}
				\E \tau_v \leq \dfrac{C }{\sqrt{\kappa A}} \,.
			\end{equation}
		\item
			In Lemma~\ref{l:hcpl}, the horizontal hitting time in~\eqref{e:tauHbound} needs to be replaced with
			\begin{equation}\label{e:tauHboundPrime}
				\tag{\ref{e:tauHbound}$'$}
				\E \tau_h \le \dfrac{C }{\sqrt{\kappa A}} \,.
			\end{equation}
	\end{enumerate}
	Once these are established, we obtain the coupling time bound
	\begin{equation}\label{e:couplingBoundPrime}
		\tag{\ref{e:couplingBound}$'$}
		\E^{(x, \tilde x)} \tau_\cpl \leq C \paren[\Big]{ \dfrac{\ep^2 }{\ka} + \dfrac{1}{\sqrt{\kappa A}}  }\,.
	\end{equation}
	instead of~\eqref{e:couplingBound}.
	Using~\eqref{e:tMixTCouple} this implies
	\begin{equation}
		\tmix \leq C \paren[\Big]{ \dfrac{\ep^2 }{\ka} + \dfrac{1}{\sqrt{\kappa A}}  } \,,
	\end{equation}
	which implies~\eqref{e:tdisBound} when $A \leq \kappa / \epsilon^4$.
\end{proof}

\section{Relationship between the dissipation time and mixing time}\label{s:tdismix}

In this section we prove Proposition~\ref{p:tMixTDis} which relates the mixing time and the dissipation time of general incompressible flows.
Throughout this section we will assume~$u \in L^\infty( [0, \infty); W^{1, \infty}(\T^d) )$ is a divergence free vector field, and let~$X$ be the (time inhomogeneous) Markov process on $\T^d$ defined by the SDE
\begin{equation}\label{e:sdeXu}
  dX_t = u(X_t) \, dt + \sqrt{\kappa} \, dB_t\,.
\end{equation}
Here $B$ is a standard $d$-dimensional Brownian motion on the torus.

Let $\rho(x, s; y, t)$ be the transition density of~$X$.
By the Kolmogorov equations, we know that~$\rho$ is the fundamental solution to~\eqref{e:phiEq}, and thus the mixing time of~$X$ is given by~\eqref{e:tmixDef}.
Using the Kolmogorov equations again, the dissipation time (defined in~\eqref{e:tdisDefPDE}) can be equivalently defined by
\begin{equation}\label{e:tdisDef}
  \tdis \defeq  \inf\set[\Big]{ t \geq 0 \st \sup_{x \in \T^d,\; s \geq 0} \norm{ \E^{(\cdot,s)} \vartheta(X_{s+t}) }_{L^2} < \frac{1}{2} \norm{\vartheta}_{L^2} \,,~\forall \vartheta\in \dot L^2(\T^d)  }\,.
\end{equation}
Recall $\dot L^2(\T^d)$ is the set of all mean zero $L^2$ functions on the torus $\T^d$, and $\E^{(x,s)}$ denotes the expected value under the probability measure $\P^{(x,s)}$ under which $X_s = x$ almost surely.
The constant~$1/2$ in~\eqref{e:tdisDefPDE}, \eqref{e:tmixDef} and~\eqref{e:tdisDef} is chosen for convenience.
Replacing it by any constant that is strictly smaller than~$1$ will only change $\tmix$ and~$\tdis$ by a constant factor that is independent of~$u$, $\kappa$ and $d$.

The first inequality in~\eqref{e:tmixDef} can be proved elementarily, and we do that first.
\begin{lemma}\label{l:disMix}
  The dissipation time and mixing time satisfy the inequality
	\begin{equation*}
		\tdis
			\leq 3 \tmix\,.
	\end{equation*}
\end{lemma}
\begin{proof}
  For simplicity, and without loss of generality, we will assume that $s = 0$ in both~\eqref{e:tmixDef} and~\eqref{e:tdisDef}.
  Let $\theta_t(x) \defeq \E^{(x,0)} \theta_0(X_t) = \E^x \theta_0(X_t)$ for some $\theta_0$ in $\dot L^2$.
  Since $\theta_0$ is mean~$0$, we note
  \begin{equation*}
    \theta_t(x)
      = \int_{\T^d} \rho(x, 0; y, t) \theta_0(y) \, dy
      = \int_{\T^d} (\rho(x, 0; y, t) - 1) \theta_0(y) \, dy \,,
  \end{equation*}
  and hence
  \begin{align}\label{e:thetaL21}
    \nonumber
    \norm{\theta_t}_{L^2}^2
      &\leq \paren[\Big]{
	  \int_{\T^d \times \T^d} \abs{ \rho(x, 0; y, t) - 1 } \, dy \, dx
	}
      \cdot
    \\
      &\qquad\qquad
	\paren[\Big]{
	    \int_{\T^d \times \T^d } \theta_0(y)^2 (\rho(x, 0; y, t) + 1) \, dx \, dy
	  }\,.
  \end{align}
  Since the Lebesgue measure is invariant, we note
  \begin{equation*}
    \int_{\T^d} \rho(x, 0; y, t) \, dx = 1\,,
    \qquad\text{for every }y \in \T^d\,,
  \end{equation*}
  and hence~\eqref{e:thetaL21} implies
  \begin{equation}\label{e:thetaL2}
  \norm{\theta_t}_{L^2}^2 \leq 2 \norm{\theta_0}_{L^2}^2 \paren[\Big]{
	  \sup_{x \in \T^d} \int_{\T^d} \abs{ \rho(x, 0; y, t) - 1 } \, dy
	}\,.
  \end{equation}
  
  The Chapman Kolmogorov equations and invariance of the Lebesgue measure immediately imply
  \begin{equation}\label{e:ntmix}
    \sup_x \int_{\T^d} \abs{\rho(x, 0; y, n \tmix ) - 1} \leq \frac{1}{2^n}\,,
  \end{equation}
  for any natural number~$n \in \N$.
  Choosing $t = 3 \tmix$ and using~\eqref{e:thetaL2}, we see that~\eqref{e:ntmix} immediately implies $\norm{\theta_{3\tmix}}_{L^2}^2 \leq \frac{1}{4} \norm{\theta_0}_{L^2}^2$, which finishes the proof.
\end{proof}

The proof of the second inequality in~\eqref{e:disMix} follows from Proposition~4.2 in~\cite{IyerXuEA21}, which provides an~$L^1$ to $L^\infty$.
We reproduce this here for convenience, and then go on to prove the second inequality in~\eqref{e:disMix}.

\begin{lemma}[Proposition 4.2 in~\cite{IyerXuEA21}]\label{l:L1Linf}
  There exists a constant $C = C(d)$, independent of $u$, such that for all $\vartheta \in \dot L^2$, and all sufficiently small~$\kappa > 0$ we have
  \begin{equation}\label{e:L1Linf}
    \norm{\E^{(x, s)} \vartheta(X_{s+t}) }_{L^\infty} \leq \frac{1}{2} \norm{\vartheta}_{L^1}\,,
    \quad\text{for all }
    s \geq 0\,,~
    t \geq C \tdis \ln \paren[\Big]{1 +  \frac{1}{\kappa \tdis}  } \,.
  \end{equation}
\end{lemma}
\begin{remark}\label{r:L1LinfN}
  Since $\norm{\vartheta}_{L^1} \leq \norm{\vartheta}_{L^\infty}$, we can iterate~\eqref{e:L1Linf} to yield
  \begin{equation*}
    \norm{\E^{(x, s)} \vartheta(X_{s+t}) }_{L^\infty} \leq 2^{-n} \norm{\vartheta}_{L^1}\,,
    \quad\text{for all }
    s \geq 0\,,~
    t \geq n C \tdis \ln \paren[\Big]{1 +  \frac{1}{\kappa \tdis}  } \,.
  \end{equation*}
\end{remark}
\begin{proof}
  For simplicity, and without loss of generality we assume $s = 0$.
  Using well known drift independent estimates (see for instance Lemma~5.6 in~\cite{ConstantinKiselevEA08}, Lemmas 3.1, 3.3 in~\cite{FannjiangKiselevEA06}, and Lemma 5.4 in~\cite{Zlatos10}) we know
  \begin{equation*}
    \norm{\theta_{t + 2\tdis}}_{L^\infty}
      \leq \frac{c}{(\kappa \tdis)^{\frac{d}{4}}} \norm{\theta_{t + \tdis}}_{L^2} \,,
    \qquad\text{and}\qquad
    \norm{\theta_{\tdis}}_{L^2} \leq \frac{c}{(\kappa \tdis)^{\frac{d}{4}} } \norm{\theta_0}_{L^1}\,,
  \end{equation*}
  for some dimensional constant $c = c(d)$.
  Now iterating~\eqref{e:tdisDef} we see
  \begin{equation*}
      \norm{\theta_{t + \tdis}}_{L^2}
	\leq 2^{- \floor{t/\tdis} } \norm{\theta_{\tdis}}_{L^2} 
	\leq 2^{1 - t/\tdis } \norm{\theta_{\tdis}}_{L^2} \,,
  \end{equation*}
  and hence
  \begin{equation*}
    \norm{\theta_{t + 2t^*_p}}_{L^\infty}
      \leq \frac{c}{(\kappa \tdis)^{\frac{d}{4}}} \norm{\theta_{t + \tdis}}_{L^2}
      \leq \frac{c 2^{1 - t/\tdis}}{(\kappa \tdis)^{\frac{d}{4}}} \norm{\theta_{\tdis}}_{L^2}
      \leq \frac{c^2 2^{1 - t/\tdis}}{(\kappa \tdis)^{\frac{d}{2}}} \norm{\theta_0}_{L^1}
  \end{equation*}
  Thus if we choose $t \geq C \tdis \ln ( 1 + 1/(\kappa \tdis) )$ for some sufficiently large constant $C = C(p, d)$, we obtain
  \begin{equation*}
    \norm{\theta_{t + 2\tdis}}_{L^\infty} \leq \frac{1}{2} \norm{\theta_0}_{L^1}\,.
  \end{equation*}
  This finishes the proof of Lemma~\ref{l:L1Linf}.
\end{proof}

We can now prove the second inequality in~\eqref{e:disMix}.
\begin{lemma}\label{l:mixDis}
  There exists a dimensional constant $C = C(d)$, independent of~$u$ and~$\kappa$ such that
	\begin{equation*}
		\tmix
			\leq C \tdis \ln \paren[\Big]{ 1 + \frac{1}{\ka \tdis} }\,.
	\end{equation*}
\end{lemma}
\begin{proof}
  For simplicity, and without loss of generality we assume $s = 0$.
  Choose large enough so that~\eqref{e:L1Linf} holds.
  By standard regularity theory, we know that for any $\epsilon > 0$, the density $\rho(x, 0; y, \epsilon)$ is integrable in~$y$.
  Since $\rho \geq 0$, we note
  \begin{equation*}
    \int_{\T^d} \abs{ \rho(x, 0; y, \epsilon) - 1 } \, dy
      \leq \int_{\T^d} \paren{ \rho(x, 0; y, \epsilon) + 1 } \, dy
      = 2\,.
  \end{equation*}
  Let $C$ be the constant from~\eqref{e:L1Linf} and choose
  \begin{equation*}
    t = 2 C \tdis \ln \paren[\Big]{ 1 + \frac{1}{\kappa \tdis} } \,.
  \end{equation*}
  Iterating Lemma~\ref{l:L1Linf} (as in Remark~\ref{r:L1LinfN}), we note that for every $x \in \T^d$
  \begin{align*}
    \norm{\rho(x, 0; y, 2t + \epsilon) - 1}_{L^1(y)}
      &\leq \norm{\rho(x, 0; y, 2t + \epsilon) - 1}_{L^\infty(y)}
      \\
      &\leq \frac{1}{4} \norm{\rho(x, 0; y, \epsilon) - 1}_{L^1(y)}
      \leq \frac{1}{2} \,.
  \end{align*}
  This shows $\tmix \leq t$, finishing the proof.
\end{proof}

The proof of Proposition~\ref{p:tMixTDis} follows immediately from Lemmas~\ref{l:disMix} and~\ref{l:mixDis}.

\section{Coupling of Projections (Proof of Lemma~\ref{l:cellCoupling})}\label{s:cellCoupling}

In this section we prove Lemma~\ref{l:cellCoupling} showing that the projected processes $Y$, $\tilde Y$ (defined in~\eqref{e:Ydef}) will couple in time $O(\epsilon^2  / \kappa)$ in expectation.
Coupling of diffusions have been studied by many authors, dating back to Lindvall and Rogers~\cite{LindvallRogers86}.
In their original work, Lindvall and Rogers~\cite{LindvallRogers86} provide a method to couple diffusions in $\R^d$ by ``reflecting'' the noise.
Unfortunately, if we use their methods directly the bound we obtain on the coupling time will depend on the Lipschitz constant of the drift; in our case, this is $O(A/\ep)$ which is unbounded.
It is for this reason that we modify the cellular flows using the cutoff function~$\xi$.
With the cutoff, we have a central region in each square where there is no drift.
Once $Y, \tilde Y$ enter this region, they can be successfully coupled by reflection.

To carry out the details of the above plan, define
\begin{equation*}
  Q = [0, \epsilon/2]^2\,,
  \quad
  U = Q \cap \set{H > 1/2}\,,
  \quad
  U' = Q \cap \set{H > h_0}\,,
\end{equation*}
for some ${h_0 \in (3/4, 1)}$ that is independent of $\epsilon$, $A$, $\kappa$ and will be chosen shortly.
We will run $Y$ and $\tilde Y$ independently until they both enter~$U'$, and then reflect the noise until they couple.
To estimate the time taken by each of these steps we use the following results.

\begin{lemma}\label{l:crudeMixTimeBd}
  Let $u \in L^\infty( [0, \infty); W^{1, \infty}(\T^d))$ be a general (not necessarily cellular) divergence free drift, and consider the SDE~\eqref{e:sdeXu} on the $d$-dimensional torus $\T^d$.
	The mixing time of~$X$ is bounded by
	\begin{equation}\label{e:tmixX}
		\tmix(X) \leq \frac{C}{\kappa}\,,
	\end{equation}
	for some dimensional constant~$C = C(d)$ that is independent of~$u$ and~$\kappa$.
\end{lemma}
\begin{remark}\label{r:crudeMixTimeBdRescaled}
  By rescaling, on a torus with side length~$\ell$, the bound~\eqref{e:tmixX} becomes
	\begin{gather}\label{e:tmixXprime}
		\tag{\ref*{e:tmixX}$'$}
		\tmix(X) \leq \frac{C \ell^2}{\kappa}\,,
	\end{gather}
	for some dimensional constant~$C$ that is independent of~$\ell$,~$u$ and~$\kappa$.
\end{remark}
\begin{remark}
  We believe that in this generality there exists a coupling for which $\E \tcpl \leq C / \kappa$, however we are presently unable to produce such a coupling.
\end{remark}

\begin{lemma}\label{l:mix1cell}
  Let $\tilde B$ be a Brownian motion that is independent of~$B$.
  There exists a time $t_1 \leq O(\ep^2 /\kappa )$ such that for all $t \ge t_1$, we have
  \begin{equation}\label{e:mix1cell}
    \inf_{y, \tilde y \in \T^2_\ep} \P^{(y, \tilde y)} \paren{Y_t,~ \tilde Y_t \in U'} \ge \frac{\abs{U'}^2}{4\epsilon^4}. 
  \end{equation}
  Here~$\abs{U'} = \leb(U')$ denotes the Lebesgue measure of~$U'$.
\end{lemma}

\begin{lemma}\label{l:coupleBeforeExit}
	Choose the Brownian motion~$\tilde B$ to be the Brownian motion~$B$ reflected about the line perpendicular to $y - \tilde y$.
	Explicitly, choose $\tilde B = M B$, where
	\begin{equation*}
		M = I - 2 \hat n \hat n^T\,,
		\qquad\text{and}\qquad
		\hat n = \frac{y - \tilde y}{\abs{y - \tilde y}}\,.
	\end{equation*}
  There exists a time $t_2 \leq O(\epsilon^2/\kappa)$ and a constant $c > 0$ such that
  \begin{equation*}
    \inf_{y, \tilde y \in U'} \P^{(y, \tilde y)} \paren{ \tcpl^Y \leq t_2 }  \geq c \,.
  \end{equation*}
\end{lemma}

Momentarily postponing the proofs of Lemmas~\ref{l:mix1cell} and~\ref{l:coupleBeforeExit}, we prove Lemma~\ref{l:cellCoupling}.

\begin{proof}[Proof of Lemma~\ref{l:cellCoupling}]
	Choose~$t_1$ according to Lemma~\ref{l:mix1cell} and run $Y$ and $\tilde Y$ independently until time $t_1$.
	Lemma~\ref{l:mix1cell} guarantees that at time~$t_1$ we have~\eqref{e:mix1cell}.
	(Note that $\abs{U'} = O(\ep^2)$, and so $\frac{\abs{U'}^2}{4 \ep^4} = O(1)$.)

	Now choose~$t_2$ and~$\tilde B$ according to Lemma~\ref{l:coupleBeforeExit}.
	This construction will guarantee 
	\begin{equation*}
		\inf_{y, \tilde y \in U'} \P^{(y, \tilde y)} \paren{ \tcpl^Y \geq (t_1 + t_2) } \le 1 - c'\,,
		\quad\text{where}\quad
		c' \defeq \frac{c \abs{U'}^2}{4 \epsilon^4} > 0\,,
	\end{equation*}
	and~$c$ is the constant in Lemma~\ref{l:coupleBeforeExit}.

	In the event that $\tcpl^Y > t_1 + t_2$, we simply repeat the above two steps.
	The Markov property will guarantee 
	\begin{equation*}
		\inf_{y, \tilde y \in U'} \P^{(y, \tilde y)} (\tcpl^Y \geq n (t_1 + t_2) ) \le (1-c')^n\,.
	\end{equation*}
	Thus, for any $y, \tilde y \in U'$ we see
	\begin{align*}
		\E^{(y, \tilde y)} \tcpl^Y
			&= \int_0^\infty \P^{(y, \tilde y)} ( \tcpl^Y \geq t ) \, dt
			\leq (t_1 + t_2) \sum_{n = 0}^\infty  \P^{(y, \tilde y)} ( \tcpl^Y \geq n(t_1 + t_2) )
		\\
			&\leq (t_1 + t_2) \sum_{n = 0}^\infty  (1 - c')^n
			\leq \frac{t_1 + t_2}{c'}\,,
	\end{align*}
	concluding the proof.
\end{proof}

It remains to prove Lemmas~\ref{l:crudeMixTimeBd}, \ref{l:mix1cell} and~\ref{l:coupleBeforeExit}.

\begin{proof}[Proof of Lemma~\ref{l:crudeMixTimeBd}]
	Let~$\phi$ solve~\eqref{e:phiEq}.
	Multiplying by~$\phi$, using the fact that $\dv u = 0$ and the Poincar\'e inequality shows
	\begin{equation*}
		\norm{\phi(s + t)}_{L^2} \leq e^{\lambda_1 \kappa t} \norm{\phi(s)}_{L^2}\,,
	\end{equation*}
	where~$\lambda_1$ is the first non-zero eigenvalue of the Laplacian on $\T^d$.
	This immediately implies~$\tdis(X) \leq 1 / (\lambda_1 \kappa)$, and using Proposition~\ref{p:tMixTDis} concludes the proof.
\end{proof}

\begin{proof}[Proof of Lemma \ref{l:mix1cell}]
	Using Lemma~\ref{l:crudeMixTimeBd} and rescaling (see Remark~\ref{r:crudeMixTimeBdRescaled}), we see
	\begin{equation*}
	  \tmix(Y) \leq \frac{C \epsilon^2}{\kappa}\,.
	\end{equation*}
	Now choose $N = \log_2\paren{ 2\epsilon^2/\abs{U'} }$, and note that $N$ is independent of~$\epsilon$.
	For every $t \geq N \tmix(Y)$ and $y \in \T^2_\epsilon$, we have
	\begin{equation*}
		\abs[\Big]{
			\P^y (Y_t \in U')
				- \frac{\abs{U'}}{\epsilon^2}
		}
		\leq 2^{-N} \frac{\abs{U'}}{2\epsilon^2}
		\quad\text{and hence}\quad
			\P^y (Y_t \in U')
			\geq \frac{\abs{U'}}{2\epsilon^2}\,,
	\end{equation*}
	Since $\tilde Y$ is independent of~$Y$ and satisfies the same bound we obtain~\eqref{e:mix1cell} as claimed.
\end{proof}

\begin{proof}[Proof of Lemma \ref{l:coupleBeforeExit}]
	Notice that as long as $Y$, $\tilde Y$ remain in~$U$, they are simply rescaled standard Brownian motions.
	Let $\ell_b$ be perpendicular bisector of the line segment joining~$y$ and $\tilde y$, and $\tau_\ell$ be the hitting time of~$Y$ to $\ell_b$.
	The choice of~$\tilde B$ ensures that if $Y, \tilde Y$ remain inside~$U$ until they $\ell_b$, then they couple at time~$\tau_\ell$.

	In order to estimate the hitting time to~$\ell_b$ before exiting~$U$, let $R = \abs{y - \tilde y} / 2$ be the distance of $y$ to $\ell_b$.
	Let $K$ be the square with center $y$, side length $2R$, and one pair of sides parallel to $\ell_b$.
	Note that if $h_0$ is sufficiently closed to $1$, this square lies entirely in $U$. 
	Let $\tau_K$ be the exit time of $Y$ from~$K$, and note that
	\begin{equation*}
	  \P^{y, \tilde y} \paren{ \tcpl^Y \leq t }
		\geq \P^y \paren{ \tau_K \leq t, ~ Y_{\tau_K} \in \ell_b }
		= \frac{1}{4} \P^y \paren{ \tau_K \leq t } \,.
	\end{equation*}
	The last equality followed by symmetry, as at time $\tau_K$ it is equally likely that $Y_{\tau_K}$ belongs to any of the four sides of~$K$.

	The last term on the right can be bounded by Chebyshev's inequality, and the fact that the expected exit time of Brownian motion from a square is known.
	Namely,
	\begin{equation*}
	  \P^y(\tau_K \leq t)
			\geq 1 - \frac{\E^y \tau_K}{t}
			\geq 1 - \frac{R^2}{\kappa t}
			\geq 1 - \frac{\epsilon^2}{\kappa t}
			\geq \frac{1}{2}\,,
	\end{equation*}
	provided $t \geq 2 \epsilon^2 / \kappa$.
	Choosing $t_2 = 2 \epsilon^2 / \kappa$ concludes the proof.
\end{proof}

\section{Synchronization and Reflection}\label{s:phase2}

In this section we prove Lemmas~\ref{l:vhit}--\ref{l:hcpl} when $A \geq \kappa / \epsilon^4$, and the modified estimates~\eqref{e:vcplPrime}, \eqref{e:tauHboundPrime} when $A < \kappa / \epsilon^4$.

\subsection{Boundary layer crossings when \texorpdfstring{$A \geq \kappa / \epsilon^4$}{A > k/e4}}

In order to prove Lemmas~\ref{l:vhit}--\ref{l:hcpl}, we will need bounds on the boundary layer crossing time.
These have been studied previously by various authors (see for instance~\cite{FreidlinWentzell12,IyerNovikov16,HairerIyerEA18}), and the version we quote here can be obtained by a direct rescaling of those in~\cite{IyerNovikov16}.
Define the \emph{boundary layer} by $\b_\delta$ by
\begin{equation}\label{e:Bdelta}
  \b_\delta \defeq \{\abs{H} < \delta\} \subset \T^2\,,
\end{equation}
where we recall from~\eqref{e:tdisBound} that~$\delta = \sqrt{\kappa / A}$.
The middle of the boundary layer is the level set~$\set{H = 0}$, and is known as the \emph{separatrix}.

We will now study repeated exits from the boundary layer, followed by returns to the separatrix.
Define the sequences of stopping times~$\sigma_n$ and~$\tau_n$ inductively by starting with $\sigma_0 = 0$ and $\tau_0 = \inf \{t \ge 0 \st H(X_t) = 0\}$.
For $n \geq 1$, define
\begin{align}
  \label{def:sigma-n}
  \sigma_n &= \inf \{t \geq \tau_{n-1} \st X_t \notin \b_\delta\}
  \\
  \label{def:tau-n}
  \tau_n &= \inf \{t \geq \sigma_n \st H(X_t) = 0\}\,.
\end{align}
That is,~$\sigma_n$ is the first exit from the boundary layer~$\b_\delta$ after time~$\tau_n$, and $\tau_n$ is the first return to the separatrix after time~$\sigma_n$.

At time~$\tau_n$ we must have either~$X^1 \in \frac{\epsilon}{2} \Z$, or~$X^2 \in \frac{\epsilon}{2} \Z$.
We now separate the times when $X^1 \in \frac{\epsilon}{2} \Z$, and when~$X^2 \in \frac{\epsilon}{2} \Z$.
Given~$i \in \set{1, 2}$, let~$\tau^i_0 = 0$ and inductively define
\begin{equation*}
  \tau_n^i
    = \inf \set[\Big]{\tau_k > \tau_{n-1}^i \st  X^i_{\tau_k} \in \frac{\ep}{2} \Z }\,.
\end{equation*}
We claim that up to a logarithmic factor, the chance that $\tau_n^i \leq t$ is comparable to the number of crossings of a standard Brownian motion over an interval of size~$\epsilon / \sqrt{A}$.
This is the first lemma we state. 
\begin{lemma}\label{l:nthhits}
	Suppose $A \gg \kappa / \epsilon^2$.
  There exists a constant $c > 0$ such that, for $n \in \N$, $i \in \{1, 2\}$, we have
  \begin{equation}
    \label{eq:nthhits}
    \inf_{\abs{H(x)} < \delta} \P^x \paren{\tau^i_n \le t} \geq \left( 1 - \dfrac{c n \ep \delta \abs{\ln \delta}}{\sqrt{\ka t}} \right)^+.
  \end{equation}
\end{lemma}

This lemma is simply a rescaling of Lemma~2.2 in~\cite{IyerNovikov16}, and we refer the reader there for the proof.
While it applies whenever $A \gg \kappa/\epsilon^2$, it only yields the optimal bound when $A \geq \kappa / \epsilon^4$, and this is the only case we apply it in.
(The case $A < \kappa / \epsilon^4$ is covered in Section~\ref{s:2nd-case-prob-bound}, below.)

The proof of Lemma~\ref{l:nthhits} in~\cite{IyerNovikov16} uses PDE techniques from~\cite{ChildressSoward89,FannjiangPapanicolaou94,NovikovPapanicolaouEA05,IyerKomorowskiEA14}.
A proof of Lemma~\ref{l:nthhits} can be obtained by directly using probabilistic techniques, and similar estimates appeared in~\cite{HairerKoralovEA16,HairerIyerEA18}.

In order to apply Lemma~\ref{l:nthhits}, we need the process~$X$ to enter the boundary layer~$\b_\delta$.
This happens in time at most $O(\epsilon^2 / \kappa )$, and is the content of our next lemma.
\begin{lemma}\label{l:exit-int}
	Let $\sigma_e$ be the first hitting time of $X$ to the level set $\set{H = \delta}$ (i.e.\ $\sigma_e = \inf \{t \geq 0 \mid H(X_t) = \delta\}$).
	Then
	\begin{equation}
		\sup_{x \in \T^2} \E^{x} \sigma_e \le \frac{C \ep^2}{\ka} \,.
	\end{equation}
\end{lemma}
\begin{proof}
	We first project to the torus of side length~$\epsilon$, and note that $\sigma_e = \inf\set{ t \geq 0 \st H(Y_t) = \delta }$.
	Note $\set{ H > \delta }$ contains two connected components, each occupying an area of at most $1/4$ of the torus $\T^2_\epsilon$.
	For any $x \in \T^2_\epsilon$, let $U$ be the connected component of $\set{H > \delta}$ that contains~$x$.
	Thus, for any $t \geq \tmix(Y)$ we know
	\begin{equation*}
		\abs[\Big]{ \P^{x} (X_t \in U^c)  - \frac{\abs{U^c}}{\epsilon^2} } \leq \frac{1}{4} \,,
		\quad\text{and hence}\quad
		\P^{x_0} ( X_t \in U)  \geq \frac{1}{2}\,.
	\end{equation*}
	By continuity of trajectories we note that the event $\set{\sigma_e \le t} \supseteq \set{ X_t \in U^c}$, and so $\P^{x_0} ( \sigma_e < t ) \geq 1/2$.
	In the event that $\sigma_e > t$, we use the Markov property and repeat the  above argument to yield
	\begin{equation*}
	  \E^{x} \sigma_e \leq 2 \tmix(Y)\,.
	\end{equation*}
	By~\eqref{e:tmixXprime} with $\ell = \epsilon$ we know $\tmix(Y) \leq C \epsilon^2 / \kappa$, concluding the proof.
\end{proof}

\subsection{Proofs of the hitting time estimates (Lemmas~\ref{l:vhit} and~\ref{l:hhit})}\label{sb:hit}

Then we may estimate the first hit at vertical boundary lines.

\begin{proof}[Proof of Lemma~\ref{l:vhit}]
  Notice that, periodicity of~$v$ and the synchronous choice~$\tilde B = B$, implies $\sigma_v = \tilde \sigma_v$.
  Thus we only have to prove~\eqref{e:vhit}.
	Without loss of generality assume $(\Pi_\epsilon X_0)^1 \in [0, \epsilon/2)$.
	(We clarify that $(\Pi_\epsilon X_0)^1$ refers to the first coordinate of $\Pi_\epsilon X_0$.)
	If $(\Pi_\epsilon X_0)^1 = 0$, then $\sigma_v = 0$, and there is nothing to prove, and thus we may assume $(\Pi_\epsilon X_0)^1 \in (0, \epsilon/2)$.
	Let $V \subseteq \T^2_\epsilon$ be the set of all points $y$ such that $y_1 \in [\epsilon/2, \epsilon]$, and note that $V$ occupies half the area of $\T^2_\epsilon$.
	Thus for any $t \geq 2\tmix(Y)$ we see
	\begin{equation*}
		\abs[\Big]{ \P ( Y_t \in V ) - \frac{1}{2} } \leq \frac{1}{4}
		\quad\text{and hence}\quad
		\P ( Y_t \in V ) \geq \frac{1}{4}\,.
	\end{equation*}
	By continuity of trajectories, $\set{ \sigma_v \leq t } \supseteq \set{ Y_t \in V }$, and so $\P ( \sigma_v \le t ) \geq 1/4$.
	If $\sigma_v > t$, then we use the Markov property and repeat the above argument to show
	\begin{equation*}
		\E \sigma_v \leq 8 \tmix(Y) \leq \frac{C \epsilon^2 }{\kappa} \,,
	\end{equation*}
	as desired.
\end{proof}

\begin{proof}[Proof of Lemma~\ref{l:hhit}]
  The proof is identical to the proof of Lemma \ref{l:vhit}.
\end{proof}

\subsection{Coupling time estimates (Lemmas~\ref{l:vcpl} and~\ref{l:hcpl})}\label{sb:cpl}

We now turn our attention to Lemma~\ref{l:vcpl}.
Note first that by definition $v$ is~$\epsilon$ periodic and
\begin{align*}
  v_1(-x_1, x_2) &= -v_1(x_1, x_2)\,,
	&
  v_2(-x_1, x_2) &= v_2(x_1, x_2)\,,
	\\
  v_1(x_1, -x_2) &= v_1(x_1, x_2)\,,
	&
  v_2(x_1, -x_2) &= -v_2(x_1, x_2)\,,
	\\
	v( x_1 + \tfrac{\epsilon}{2}, x_2 ) &= -v(x_1, x_2)\,,
	&
	v( x_1, x_2  + \tfrac{\epsilon}{2} ) &= -v(x_1, x_2)\,.
\end{align*}
As a result choosing~$\tilde B = (-B_1, B_2)$ and the assumptions $\Pi_\epsilon X_0 = \Pi_\epsilon \tilde X$ and $X_0^1, \tilde X_0^1 \in \frac{\ep}{2} \Z$ imply
\begin{equation}\label{e:symmetry}
  X_t^1 = -\tilde X_t^1 \pmod{\epsilon}\,,
	\quad\text{and}\quad
  X_t^2 = \tilde X_t^2 \pmod{\epsilon}\,.
\end{equation}

Let
\begin{equation*}
	\ell_1 \defeq \set[\Big]{ \frac{X^1_0 + \tilde X^1_0}{2} }\times \T^1 \subseteq \T^2
\end{equation*}
be the vertical line half way between $X_0$ and $\tilde X_0$.
Note $(X^1_0 + \tilde X^1_0) / 2 \in \frac{\ep}{2} \Z$ and so~$\ell_1$ is contained in the separatrix.
By~\eqref{e:symmetry} we see that~$\tau_v$ is exactly the hitting time of~$X$ to $\ell_1$.
Thus we may now ignore~$\tilde X$ and simply estimate the hitting time of~$X$ to~$\ell_1$.

Note that $X_{\tau^1_n} \in (\frac{\epsilon}{2} \Z) \times \R $ for all $n$ and behaves like a random walk on the collection of vertical lines~$(\frac{\epsilon}{2} \Z) \times \T^1 \subseteq \T^2$.
There are $2/\epsilon$ such vertical lines in the torus $\T^2$, and so we expect that after $O(1/\epsilon^2)$ steps of this random walk, $X_{\tau^1_n}$ will land in our desired line segment~$\ell_1$.
This is our next result.

\begin{lemma}\label{l:tauvN-hit-ell1}
	Note that $\tau_v = \inf \{t \ge 0 \st X_{t} \in \ell_1\}$.
	There exists $p_0 > 0$, and a constant $C_1$, independent of $A, \ep, \ka$, such that, for $n = C_1 / \ep^2$, and $x \in \T^2$ such that $x_1 \in \frac{\ep}{2} \Z$, 
	\begin{equation*}
		\P^x \paren{ \tau_v \le \tau^1_n }
			\geq p_0\,.
	\end{equation*}
\end{lemma}

Postponing the proof of Lemma~\ref{l:tauvN-hit-ell1} to Section~\ref{s:pullback}, we prove Lemma~\ref{l:vcpl}.

\begin{proof}[Proof of Lemma~\ref{l:vcpl}]
	As explained above, $\tau_v$ is the hitting time of~$X$ to the bisector~$\ell_1$.
	Using Lemmas~\ref{l:nthhits} and~\ref{l:tauvN-hit-ell1} we see that
	\begin{equation}
		\label{e:froml6.1}
	  \P (\tau_n^1 \leq t_1) \geq 1 - \frac{p_0}{2}\,,
		\quad 
		\text{where }
		t_1 = \frac{4c^2 C_1^2 \abs{\ln \delta}^2 }{p_0^2 \epsilon^2 A}
		\,,
		\quad
		n = \frac{C_1}{\epsilon^2}\,.
	\end{equation}
	Here $c$ is the constant from equation~\eqref{eq:nthhits}, and $p_0$,~$C_1$ are constants from Lemma~\ref{l:tauvN-hit-ell1}.
	With Lemma~\ref{l:tauvN-hit-ell1}, we also see that
	\begin{equation}
		\label{e:froml6.3}
		\P^x (\tau_v \le \tau^1_n) \ge p_0 \,.
	\end{equation}
	Combining~\eqref{e:froml6.1} and~\eqref{e:froml6.3} gives 
	\begin{equation*}
		\P^x \paren[\big]{\tau_v \le \tau^1_n \leq t_1 }  \geq \frac{p_0}{2} \,,
	\end{equation*}
	which implies $\P ( \tau_v \leq t_1 ) \geq \frac{p_0}{2}$.
	Using the Markov property and iterating this implies
	\begin{equation*}
		\P (\tau_v > k t_1)  \leq \paren[\Big]{1 - \frac{p_0}{2} }^k\,,
		\quad\text{and hence }
		\E \tau_v \leq \frac{2 t_1}{p_0}\,,
	\end{equation*}
	finishing the proof.
\end{proof}

\begin{proof}[Proof of Lemma~\ref{l:hcpl}]
  The proof is identical to the proof of Lemma \ref{l:vcpl}.
	Note that at times when $X^2_t = \tilde X^2_t$, we actually have $X_t = \tilde X_t$ and hence $X_{\tau_h} = \tilde X_{\tau_h}$.  
\end{proof}

\subsection{The hitting time to the bisector (Lemma~\ref{l:tauvN-hit-ell1})}\label{s:pullback}

In order to prove Lemma~\ref{l:tauvN-hit-ell1} we will lift trajectories of $X$ from the torus $\T^2$ to the covering space $\R^2$.
For clarity, we will denote the lifted process by $\hat X$.
Define the family of lines
\begin{equation*}
	\hat \ell_1 = \set[\Big]{ x \in \R^2 \st x_1 = n + \frac{n_0\epsilon }{2}\,, n \in \Z }\,,
\end{equation*}
where $n_0 \in \Z$ is chosen such that
\begin{equation*}
\ell_1 = \set[\Big]{ x \in \T^2 \st x_1 = \frac{n_0 \epsilon }{2} }\,.
\end{equation*}
Note that the event of $X$ hitting $\ell_1$ on $\T^2$ is exactly the same as the event of $\hat X$  hitting $\hat \ell_1$ on $\R^2$.
Moreover, if $\hat X$ travels a horizontal distance of at least $1$, then it must pass through one of the lines in~$\hat \ell_1$.
We will use this to estimate $\P(\tau_v \leq \tau_n^1)$.

\begin{lemma}\label{l:pullback}
	Suppose $\hat X$ satisfies the SDE~\eqref{eq:itodef} in~$\R^2$, with $\hat X_0 = \hat x \in \R^2$ such that $\hat x_1 = 0$.
	There exist constants $C_1$, $p_0 > 0$, independent of $A, \ep, \ka$, such that, for $n = C_1/\ep^2$ we have
	\begin{equation}\label{e:XhatTaun}
		\P^{\hat x} \paren{ \abs{\hat X_{\tau^1_n}^1} > 1 } \ge p_0 \,.
	\end{equation}
\end{lemma}

\begin{proof}
	Let $S_n = \hat X^1_{\tau^1_n}$, and observe that by symmetry of $v$ we must have $\E^{\hat x} S_n = 0$.
	If $\E^{\hat x} S_n^2 \geq 1$ we note
	\begin{equation*}
		\paren[\Big]{ \E^{\hat x} S_n^2 - 1 }^2
			\leq \paren[\Big]{ \E^{\hat x}  S_n^2 \one_{\set{ \abs{ S_n} \geq 1 }} }^2
			\leq \E^{\hat x}  S_n^4 \P^{\hat x} \paren{ \abs{S_n} \geq 1 }
	\end{equation*}
	and hence
	\begin{equation}\label{e:petrovBound}
		\P^{\hat x} \paren{ \abs{S_n} \geq 1 } \ge \frac{(\E^{\hat x}  S_n^2  - 1)^2}{\E^{\hat x} S_n^4}\,.
	\end{equation}
	whenever $\var(S_n) > 1$. 

	To use~\eqref{e:petrovBound}, we need to show~$\E^{\hat x} S_n^2 \geq 1$, and find a suitable upper bound for~$\E^{\hat x} S_n^4$.
	For the first part we note~\cite{IyerNovikov16} shows that the variance of $S_n$ is comparable to that of a random walk with steps of size~$\epsilon$.
	That is, we know
	\begin{equation}\label{e:varBound}
		\E^{\hat x} S_n^2 \ge c_1 n \ep^2 \,,
	\end{equation}
	for some constant $c_1 > 0$, that is independent of $\epsilon$, $A$ and $\kappa$.
	Thus choosing $n = 2 / (c_1 \ep^2)$ will guarantee $\E^{\hat x} S_n^2 \geq 1$.
	
	For the second part we need to find an upper bound for $\E S_n^4$. 
	For simplicity, let $\xi_m = \hat X^1_{\tau^1_{m+1}} - \hat X^1_{\tau^1_m}$, with $\tau^1_0 = 0$, so that $S_n = \xi_0 + \cdots + \xi_{n-1}$. 
	Notice
	\begin{equation*}
		\E^x S_n^4
			= \sum_{m=0}^{n-1} \E^x \abs{\xi_m}^4
				+  6 \sum_{m' = 1}^{n-1} \sum_{m = 0}^{m'-1} \E^x  \abs{\xi_m}^2 \abs{\xi_{m'}}^2  \,,
	\end{equation*}
	since the cross terms vanish by symmetry.

	From Lemma~2.1 in~\cite{IyerNovikov16}, we know that
	\begin{equation*}
		\E^{\hat x} \abs{\xi_m}^2 \le c_2 \ep^2 \,,
	\end{equation*}
	for some finite constant $c_2$ that is independent of~$\epsilon$, $A$ and $\kappa$.
	The same proof (Section~5 in~\cite{IyerNovikov16}) also shows that
	\begin{equation*}
		\E^{\hat x} \abs{\xi_m}^4 \le c_2 \ep^4 \,.
	\end{equation*}
	Moreover, for $m < m'$, by tower property, 
	\begin{equation*}
		\E^{\hat x} \paren{ \abs{\xi_m}^2 \abs{\xi_{m'}}^2 }
			= \E^{\hat x} \paren[\Big]{ \abs{\xi_m}^2 \E^{\hat X_{\tau^1_{m+1}}} \abs{\xi_{m'}}^2 }
			\le \E^x \left[ \abs{\xi_m}^2 c_2 \ep^2 \right]  \le (c_2 \ep^2)^2 \,.
	\end{equation*}
	Thus
	\begin{equation}
		\label{e:4mnt}
		\E^{\hat x} S_n^4 \le c_2 n \ep^4 + c_2^2 n^2 \ep^4 \le c n^2 \ep^4 \,,
	\end{equation}
	where $c = 2 c_2 (1 + c_2)$. 

	Combining~\eqref{e:petrovBound}, \eqref{e:varBound} and~\eqref{e:4mnt} we see 
	\begin{equation}
		\P^{\hat x} \paren{ \abs{S_n} > 1 } \ge \frac{(c_1 n \ep^2 - 1)^2}{c n^2 \ep^4} \,.
	\end{equation}
	Choosing $n = C_1 / \ep^2$ for some large constant $C_1$, we obtain~\eqref{e:XhatTaun} as desired.
\end{proof}

Using this, we prove Lemma~\ref{l:tauvN-hit-ell1}.

\begin{proof}[Proof of Lemma~\ref{l:tauvN-hit-ell1}]
	Note that $x_1 \in \frac{\ep}{2} \Z$.
	Using symmetry and periodicity, we may, without loss of generality, assume $x_1 = 0$. 
	
	Lifting the process~$X$ to $\R^2$, we recall that when $\abs{\hat X^1_t} \geq 1$, the trajectory of $\hat X$ must have passed through one of the lines in~$\hat \ell_1$.
	This implies
	\begin{equation*}
		\P^x\paren[\big]{ \tau_v \leq \tau^1_n }
			\geq \P^{\hat x} \paren[\big]{ \abs{\hat X^1_{\tau_n^1}} \geq 1 }\,,
	\end{equation*}
	and applying Lemma~\ref{l:pullback} concludes the proof.
\end{proof}

\subsection{Boundary layer crossings when \texorpdfstring{$A \leq \ka /\ep^4$}{A<k/e2}.}\label{s:2nd-case-prob-bound}
\begin{@empty}
We now prove the crossing estimates~\eqref{e:vcplPrime} and~\eqref{e:tauHboundPrime} in the case $A \leq \kappa / \epsilon^4$.
In this case, there is a better estimate on the boundary layer crossing times than Lemma~\ref{l:nthhits}, and we state this below.

In order to use estimates from Koralov~\cite{Koralov04}, we slightly modify the definition of~$\tau_n$.
Define $\check \tau_{-1} = 0$ and inductively define $\check \tau_n$ to be the first time after~$\check \tau_{n-1}$ that~$X$ returns to the separatrix $\set{H = 0}$ after crossing one of the cell diagonals.
It is known that the process~$X_{\check \tau_k}$ essentially performs a random walk on the skeleton of $4/\epsilon^2$ cell edges (see for instance~\cite{FreidlinWentzell93,FreidlinWentzell94,Koralov04}).
In order to follow our coupling argument, we separate the times when $X_{\check \tau_n}^1 \in \frac{\epsilon}{2} \Z$ or~$X_{\check \tau_n}^2 \in \frac{\epsilon}{2} \Z$ by defining
\begin{equation}
	\check \tau_n^i
    = \inf \set[\Big]{\check \tau_k > \check \tau_{n-1}^i \st  X^i_{\check \tau_k} \in \frac{\ep}{2} \Z }\,.
\end{equation}
Thus, at time $\check \tau_n^1$, the first coordinate $X^1_{\check \tau_n}$ has essentially performed~$n$ steps of a random walk.
To prove~\eqref{e:vcplPrime}, we will use the following bound on $\check \tau_n$.
\begin{lemma}
	If $\epsilon^2 / \kappa \ll 1$ then there exists a constant $c > 0$ such that, for $n \in \N$, we have 
	\begin{equation}\label{eq:nthhits-modify}
		\inf_{H(x) = 0} \E^x \check \tau_n \le \frac{c n \ep^2}{\sqrt{\ka A}} \,.
	\end{equation}
\end{lemma}
\begin{proof}
	By rescaling the bound (20) in~\cite{Koralov04} we immediately see
	\begin{equation*}
		\E^x \check \tau_1 \le \frac{c \ep^2}{\sqrt{\ka A}}\,, \qquad \text{whenever} \quad H(x) = 0 \,.
	\end{equation*}
	Now~\eqref{eq:nthhits-modify} follows immediately from the Markov property and linearity.
\end{proof}

In order to prove~\eqref{e:vcplPrime}, we first note that $\check \tau_n^i$ also satisfies~\eqref{eq:nthhits-modify}.
This follows by the same argument in Section~4 of~\cite{IyerNovikov16}.
We now follow the proof of Lemma~\ref{l:vcpl}, with one modification.
Instead of~\eqref{e:froml6.1}, we have
\begin{equation}\label{e:from6.1Prime}
	\tag{\ref{e:froml6.1}$'$}
	\P (\check \tau_n^1 \leq t_1) \geq 1 - \frac{p_0}{2}\,,
	\quad 
	\text{where }
	t_1 = \frac{2 c C_1}{p_0 \sqrt{\kappa A} }
	\,,
	\quad
	n = \frac{C_1}{\epsilon^2}\,.
\end{equation}
Now following the proof of Lemma~\ref{l:vcpl} will yield~\eqref{e:vcplPrime} as desired.
The proof of~\eqref{e:tauHboundPrime} is similar.
\end{@empty}

\bibliographystyle{halpha-abbrv}
\bibliography{refs,preprints}

\newcommand{\etalchar}[1]{$^{#1}$}
\begin{thebibliography}{CCZW21}
\expandafter\ifx\csname url\endcsname\relax
  \def\url#1{\texttt{#1}}\fi
\expandafter\ifx\csname doi\endcsname\relax
  \def\doi#1{\burlalt{doi:#1}{http://dx.doi.org/#1}}\fi
\expandafter\ifx\csname urlprefix\endcsname\relax\def\urlprefix{URL }\fi
\expandafter\ifx\csname href\endcsname\relax
  \def\href#1#2{#2}\fi
\expandafter\ifx\csname burlalt\endcsname\relax
  \def\burlalt#1#2{\href{#2}{#1}}\fi

\bibitem[ABN21]{AlbrittonBeekieEA21}
D.~Albritton, R.~Beekie, and M.~Novack.
\newblock Enhanced dissipation and {H}\"ormander's hypoellipticity, 2021.
\newblock \doi{10.48550/ARXIV.2105.12308}.

\bibitem[Bak11]{Bakhtin11}
Y.~Bakhtin.
\newblock Noisy heteroclinic networks.
\newblock {\em Probab. Theory Related Fields}, 150(1-2):1--42, 2011.
\newblock \doi{10.1007/s00440-010-0264-0}.

\bibitem[BCZ17]{BedrossianCotiZelati17}
J.~Bedrossian and M.~Coti~Zelati.
\newblock Enhanced dissipation, hypoellipticity, and anomalous small noise
  inviscid limits in shear flows.
\newblock {\em Arch. Ration. Mech. Anal.}, 224(3):1161--1204, 2017.
\newblock \doi{10.1007/s00205-017-1099-y}.

\bibitem[BLP78]{BensoussanLionsEA78}
A.~Bensoussan, J.-L. Lions, and G.~Papanicolaou.
\newblock {\em Asymptotic analysis for periodic structures}, volume~5 of {\em
  Studies in Mathematics and its Applications}.
\newblock North-Holland Publishing Co., Amsterdam-New York, 1978.

\bibitem[CCZW21]{ColomboCotiZelatiEA21}
M.~Colombo, M.~Coti~Zelati, and K.~Widmayer.
\newblock Mixing and diffusion for rough shear flows.
\newblock 2021.
\newblock \doi{10.15781/83FC-J334}.

\bibitem[Chi79]{Childress79}
S.~Childress.
\newblock Alpha-effect in flux ropes and sheets.
\newblock {\em Phys.\ Earth Planet Inter.}, 20:172--180, 1979.

\bibitem[CKRZ08]{ConstantinKiselevEA08}
P.~Constantin, A.~Kiselev, L.~Ryzhik, and A.~Zlato{\v{s}}.
\newblock Diffusion and mixing in fluid flow.
\newblock {\em Ann. of Math. (2)}, 168(2):643--674, 2008.
\newblock \doi{10.4007/annals.2008.168.643}.

\bibitem[CS89]{ChildressSoward89}
S.~Childress and A.~M. Soward.
\newblock Scalar transport and alpha-effect for a family of cat's-eye flows.
\newblock {\em J. Fluid Mech.}, 205:99--133, 1989.
\newblock \doi{10.1017/S0022112089001965}.

\bibitem[CZDE20]{CotiZelatiDelgadinoEA20}
M.~Coti~Zelati, M.~G. Delgadino, and T.~M. Elgindi.
\newblock On the relation between enhanced dissipation timescales and mixing
  rates.
\newblock {\em Comm. Pure Appl. Math.}, 73(6):1205--1244, 2020.
\newblock \doi{10.1002/cpa.21831}.

\bibitem[DK08]{DolgopyatKoralov08}
D.~Dolgopyat and L.~Koralov.
\newblock Averaging of {H}amiltonian flows with an ergodic component.
\newblock {\em Ann. Probab.}, 36(6):1999--2049, 2008.
\newblock \doi{10.1214/07-AOP372}.

\bibitem[Fen19]{Feng19}
Y.~Feng.
\newblock {\em Dissipation Enhancement by Mixing}.
\newblock ProQuest LLC, Ann Arbor, MI, 2019.
\newblock Thesis (Ph.D.)--Carnegie Mellon University.

\bibitem[FFIT20]{FengFengEA20}
Y.~Feng, Y.~Feng, G.~Iyer, and J.-L. Thiffeault.
\newblock Phase separation in the advective {C}ahn-{H}illiard equation.
\newblock {\em J. Nonlinear Sci.}, 30(6):2821--2845, 2020.
\newblock \doi{10.1007/s00332-020-09637-6}.

\bibitem[FI19]{FengIyer19}
Y.~Feng and G.~Iyer.
\newblock Dissipation enhancement by mixing.
\newblock {\em Nonlinearity}, 32(5):1810--1851, 2019.
\newblock \doi{10.1088/1361-6544/ab0e56}.

\bibitem[FKR06]{FannjiangKiselevEA06}
A.~Fannjiang, A.~Kiselev, and L.~Ryzhik.
\newblock Quenching of reaction by cellular flows.
\newblock {\em Geom. Funct. Anal.}, 16(1):40--69, 2006.
\newblock \doi{10.1007/s00039-006-0554-y}.

\bibitem[FM22]{FengMazzucato22}
Y.~Feng and A.~L. Mazzucato.
\newblock Global existence for the two-dimensional {K}uramoto-{S}ivashinsky
  equation with advection.
\newblock {\em Comm. Partial Differential Equations}, 47(2):279--306, 2022.
\newblock \doi{10.1080/03605302.2021.1975131}.

\bibitem[FNW04]{FannjiangNonnenmacherEA04}
A.~Fannjiang, S.~Nonnenmacher, and L.~{Wo\l owski}.
\newblock Dissipation time and decay of correlations.
\newblock {\em Nonlinearity}, 17(4):1481--1508, 2004.
\newblock \doi{10.1088/0951-7715/17/4/018}.

\bibitem[FP94]{FannjiangPapanicolaou94}
A.~Fannjiang and G.~Papanicolaou.
\newblock Convection enhanced diffusion for periodic flows.
\newblock {\em SIAM J. Appl. Math.}, 54(2):333--408, 1994.
\newblock \doi{10.1137/S0036139992236785}.

\bibitem[FSW22]{FengShiEA22}
Y.~Feng, B.~Shi, and W.~Wang.
\newblock Dissipation enhancement of planar helical flows and applications to
  three-dimensional {K}uramoto-{S}ivashinsky and {K}eller-{S}egel equations.
\newblock {\em J. Differential Equations}, 313:420--449, 2022.
\newblock \doi{10.1016/j.jde.2021.12.029}.

\bibitem[FW93]{FreidlinWentzell93}
M.~I. Freidlin and A.~D. Wentzell.
\newblock Diffusion processes on graphs and the averaging principle.
\newblock {\em Ann. Probab.}, 21(4):2215--2245, 1993.
\newblock \doi{10.1214/aop/1176989018}.

\bibitem[FW94]{FreidlinWentzell94}
M.~I. Freidlin and A.~D. Wentzell.
\newblock Random perturbations of {H}amiltonian systems.
\newblock {\em Mem. Amer. Math. Soc.}, 109(523):viii+82, 1994.
\newblock \doi{10.1090/memo/0523}.

\bibitem[FW12]{FreidlinWentzell12}
M.~I. Freidlin and A.~D. Wentzell.
\newblock {\em Random perturbations of dynamical systems}, volume 260 of {\em
  Grundlehren der Mathematischen Wissenschaften [Fundamental Principles of
  Mathematical Sciences]}.
\newblock Springer, Heidelberg, third edition, 2012.
\newblock \doi{10.1007/978-3-642-25847-3}.
\newblock Translated from the 1979 Russian original by Joseph Sz{\"u}cs.

\bibitem[GZ21]{GallayZelati21}
T.~Gallay and M.~C. Zelati.
\newblock Enhanced dissipation and {T}aylor dispersion in higher-dimensional
  parallel shear flows, 2021.
\newblock \doi{10.48550/ARXIV.2108.11192}.

\bibitem[Hei03]{Heinze03}
S.~Heinze.
\newblock Diffusion-advection in cellular flows with large {P}eclet numbers.
\newblock {\em Arch. Ration. Mech. Anal.}, 168(4):329--342, 2003.
\newblock \doi{10.1007/s00205-003-0256-7}.

\bibitem[HIK{\etalchar{+}}18]{HairerIyerEA18}
M.~Hairer, G.~Iyer, L.~Koralov, A.~Novikov, and Z.~Pajor-Gyulai.
\newblock A fractional kinetic process describing the intermediate time
  behaviour of cellular flows.
\newblock {\em Ann. Probab.}, 46(2):897--955, 2018.
\newblock \doi{10.1214/17-AOP1196}.

\bibitem[HKPG16]{HairerKoralovEA16}
M.~Hairer, L.~Koralov, and Z.~Pajor-Gyulai.
\newblock From averaging to homogenization in cellular flows---an exact
  description of the transition.
\newblock {\em Ann. Inst. Henri Poincar\'e Probab. Stat.}, 52(4):1592--1613,
  2016.
\newblock \doi{10.1214/15-AIHP690}.

\bibitem[IKNR14]{IyerKomorowskiEA14}
G.~Iyer, T.~Komorowski, A.~Novikov, and L.~Ryzhik.
\newblock From homogenization to averaging in cellular flows.
\newblock {\em Ann. Inst. H. Poincar\'e Anal. Non Lin\'eaire}, 31(5):957--983,
  2014.
\newblock \doi{10.1016/j.anihpc.2013.06.003}.

\bibitem[IN16]{IyerNovikov16}
G.~Iyer and A.~Novikov.
\newblock Anomalous diffusion in fast cellular flows at intermediate time
  scales.
\newblock {\em Probab. Theory Related Fields}, 164(3-4):707--740, 2016.
\newblock \doi{10.1007/s00440-015-0617-9}.

\bibitem[IXZ21]{IyerXuEA21}
G.~Iyer, X.~Xu, and A.~Zlato\v{s}.
\newblock Convection-induced singularity suppression in the {K}eller-{S}egel
  and other non-linear {PDE}s.
\newblock {\em Trans. Amer. Math. Soc.}, 374(9):6039--6058, 2021.
\newblock \doi{10.1090/tran/8195}.

\bibitem[Kor04]{Koralov04}
L.~Koralov.
\newblock Random perturbations of 2-dimensional {H}amiltonian flows.
\newblock {\em Probab. Theory Related Fields}, 129(1):37--62, 2004.
\newblock \doi{10.1007/s00440-003-0320-0}.

\bibitem[KX16]{KiselevXu16}
A.~Kiselev and X.~Xu.
\newblock Suppression of chemotactic explosion by mixing.
\newblock {\em Arch. Ration. Mech. Anal.}, 222(2):1077--1112, 2016.
\newblock \doi{10.1007/s00205-016-1017-8}.

\bibitem[LPW09]{LevinPeresEA09}
D.~A. Levin, Y.~Peres, and E.~L. Wilmer.
\newblock {\em Markov chains and mixing times}.
\newblock American Mathematical Society, Providence, RI, 2009.
\newblock \doi{10.1090/mbk/058}.
\newblock With a chapter by James G. Propp and David B. Wilson.

\bibitem[LR86]{LindvallRogers86}
T.~Lindvall and L.~C.~G. Rogers.
\newblock Coupling of multidimensional diffusions by reflection.
\newblock {\em Ann. Probab.}, 14(3):860--872, 1986.
\newblock
  \urlprefix\url{http://links.jstor.org/sici?sici=0091-1798(198607)14:3<860:COMDBR>2.0.CO;2-V&origin=MSN}.

\bibitem[MT06]{MontenegroTetali06}
R.~Montenegro and P.~Tetali.
\newblock Mathematical aspects of mixing times in {M}arkov chains.
\newblock {\em Found. Trends Theor. Comput. Sci.}, 1(3):x+121, 2006.
\newblock \doi{10.1561/0400000003}.

\bibitem[NP22]{NobiliPottel22}
C.~Nobili and S.~Pottel.
\newblock Lower bounds on mixing norms for the advection diffusion equation in
  {$\Bbb {R}^d$}.
\newblock {\em NoDEA Nonlinear Differential Equations Appl.}, 29(2):Paper No.
  12, 32, 2022.
\newblock \doi{10.1007/s00030-021-00744-1}.

\bibitem[NPR05]{NovikovPapanicolaouEA05}
A.~Novikov, G.~Papanicolaou, and L.~Ryzhik.
\newblock Boundary layers for cellular flows at high {P}\'eclet numbers.
\newblock {\em Comm. Pure Appl. Math.}, 58(7):867--922, 2005.
\newblock \doi{10.1002/cpa.20058}.

\bibitem[PS08]{PavliotisStuart08}
G.~A. Pavliotis and A.~M. Stuart.
\newblock {\em Multiscale methods -- Averaging and homogenization}, volume~53
  of {\em Texts in Applied Mathematics}.
\newblock Springer, New York, 2008.

\bibitem[RY83]{RhinesYoung83}
P.~B. Rhines and W.~R. Young.
\newblock How rapidly is passive scalar mixed within closed streamlines?
\newblock {\em J.\ Fluid Mech.}, 133:135--145, 1983.

\bibitem[Sei20]{Seis20}
C.~Seis.
\newblock Diffusion limited mixing rates in passive scalar advection, 2020.
\newblock \doi{10.48550/ARXIV.2003.08794}.

\bibitem[Tay53]{Taylor53}
G.~Taylor.
\newblock Dispersion of soluble matter in solvent flowing slowly through a
  tube.
\newblock {\em Proc. R. Soc. Lond. A}, 219(1137):186--203, 1953.
\newblock \doi{10.1098/rspa.1953.0139}.

\bibitem[Wei19]{Wei19}
D.~Wei.
\newblock Diffusion and mixing in fluid flow via the resolvent estimate.
\newblock {\em Science China Mathematics}, pages 1869--1862, 2019.
\newblock \doi{10.1007/s11425-018-9461-8}.

\bibitem[Zla10]{Zlatos10}
A.~Zlato\v{s}.
\newblock Diffusion in fluid flow: dissipation enhancement by flows in 2{D}.
\newblock {\em Comm. Partial Differential Equations}, 35(3):496--534, 2010.
\newblock \doi{10.1080/03605300903362546}.

\end{thebibliography}
\end{document}